\documentclass[a4paper]{amsart}
\usepackage{amsmath,amsthm,amssymb,latexsym,epic,bbm,comment}
\usepackage{graphicx,enumerate,stmaryrd}
\usepackage[all,2cell]{xy}
\xyoption{2cell}

\usepackage[active]{srcltx}
\usepackage[parfill]{parskip}
\UseTwocells

\newtheorem{theorem}{Theorem}
\newtheorem{lemma}[theorem]{Lemma}
\newtheorem{remark}[theorem]{Remark}

\newtheorem{proposition}[theorem]{Proposition}
\newtheorem{example}[theorem]{Example}

\newcommand{\tto}{\twoheadrightarrow}

\font\sc=rsfs10
\newcommand{\cC}{\sc\mbox{C}\hspace{1.0pt}}

\newcommand{\cI}{\sc\mbox{I}\hspace{1.0pt}}

\newcommand{\cA}{\sc\mbox{A}\hspace{1.0pt}}

\newcommand{\cY}{\sc\mbox{Y}\hspace{1.0pt}}

\font\scc=rsfs7
\newcommand{\ccC}{\scc\mbox{C}\hspace{1.0pt}}

\begin{document}

\title[Isotypic $2$-representations]{Isotypic faithful $2$-representations\\ 
of $\mathcal{J}$-simple fiat $2$-categories}
\author{Volodymyr Mazorchuk and Vanessa Miemietz}

\begin{abstract}
We introduce the class of isotypic $2$-representations for finitary $2$-categories
and the notion of inflation of $2$-representations. Under some natural assumptions
we show that isotypic $2$-representations are equivalent to inflations of 
cell $2$-representations. 
\end{abstract}

\maketitle

\section{Introduction}\label{s0}

Classification problems are one of the driving forces in representation theory. With the advent of
$2$-representation theory, a canonical major problem is to classify all $2$-representations of 
suitable $2$-categories. Another direction is the study of uniqueness questions, which has been
pursued, for example, in \cite{CR,LW,MM1}. One natural class of $2$-categories to consider is that 
of {\em finitary} $2$-categories, which provide $2$-analogues of finite dimensional algebras. 

A systematic study of abstract $2$-representation theory of finitary $2$-categories was
initiated in \cite{MM1}--\cite{MM5}, \cite{Xa}. Inspired by \cite{BFK,BG,CR,KL,Ro,So},
among others, we especially focused on the class of {\em fiat} $2$-categories, which can be considered 
as $2$-analogues of algebras with involution. Fiat $2$-categories also appear, sometimes in disguise,
in \cite{EW,EL,MT,SS} and many other papers. The first question in representation theory is that of
classifying simple representations. Its $2$-analogue was addressed in \cite{MM5}, where 
we proposed the notion of a {\em simple transitive} $2$-representation as an appropriate 
$2$-substitute for the classical notion of a simple representation. Under certain combinatorial 
assumptions, which are satisfied for most of the inspiring examples mentioned above, we obtained
a complete classification of simple transitive $2$-representations 
up to equivalence, showing them to be {\em cell} $2$-representations, previously constructed in \cite{MM1,MM2}.

In the present article, we take the first step in the direction of classifying more complicated $2$-representations.
The starting point was an attempt to classify all transitive $2$-representations. However, it turned out
that our methods yield stronger results. We introduce the class of {\em isotypic} $2$-representations.
In full analogy with the corresponding classical notion, these are $2$-representations with only one equivalence
class of simple transitive weak Jordan-H{\"o}lder subquotients. Our main result, Theorem~\ref{thmmain}, asserts
that, under the assumption that a weakly  fiat $2$-category is $\mathcal{J}$-simple and the $2$-representation is
faithful, each such $2$-representation is obtained, up to equivalence, from a cell $2$-representation by an
easy procedure which we call {\em inflation}, see Section~\ref{s2.6}. In the special case where the $2$-category concerned is a finitary quotient of a $2$-Kac-Moody algebra, this recovers results from \cite[Section 4.3.4]{Ro2}.

Let us now briefly describe the structure of the article. In Section~\ref{s1}, we collect all necessary
preliminaries on finitary and weakly fiat $2$-categories. Section~\ref{s2} recalls the basics on $2$-representations.
It also contains a combinatorial result, Proposition~\ref{prop1}, which renders superfluous the numerical 
assumption appearing in \cite[Theorem~43]{MM1}, \cite[Section~4]{MM3}, \cite[Theorems~14,16]{MM3} 
and  \cite[Theorem~18]{MM5}. Furthermore, we introduce the notion of isotypicality of a $2$-representation
and describe the construction of inflation. The last subsection contains the statement of our main 
result, Theorem~\ref{thmmain}. Section~\ref{s3}, which constitutes the core of the paper, gives a proof of the 
main theorem in the special case of a particularly nice class of $2$-categories, the $2$-categories of projective 
functors for finite dimensional self-injective algebras. In Section~\ref{s4} we deduce the general case.
To round things off, Section~\ref{s5} contains several new constructions of finitary and fiat $2$-categories,
providing many additional examples to which our theorems apply. In an appendix, we indicate how to generalize
our previous results on fiat $2$-categories to the weakly fiat case.

\vspace{5mm}

\noindent
{\bf Acknowledgment.} A substantial part of the paper was written during a visit of the second author 
to Uppsala University, whose hospitality is gratefully acknowledged. The visit was supported by 
EPSRC grant EP/K011782/1 and by the Swedish Research Council. The first author is partially supported by the 
Swedish Research Council. The second author is partially supported by EPSRC grant EP/K011782/1. 

\section{Preliminaries on fiat $2$-categories and $2$-representations}\label{s1}

\subsection{Notation}\label{s1.1}

Throughout this paper, we work over a fixed algebraically closed field $\Bbbk$.  

By a {\em $2$-category} we mean a category which is enriched over the category of all small categories. 
Thus, a $2$-category $\cC$ consists of a collection of objects, denoted $\mathtt{i},\mathtt{j},\mathtt{k},\dots$;
for each pair $(\mathtt{i},\mathtt{j})$ of objects, a set $\cC(\mathtt{i},\mathtt{j})$ of $1$-morphisms,
denoted $\mathrm{F},\mathrm{G},\mathrm{H}, \dots$; and for each pair $(\mathrm{F},\mathrm{G})$ of 
$1$-morphisms in a fixed $\cC(\mathtt{i},\mathtt{j})$, a set 
$\mathrm{Hom}_{\ccC(\mathtt{i},\mathtt{j})}(\mathrm{F},\mathrm{G})$ of $2$-morphisms, denoted 
$\alpha,\beta,\gamma,\dots$. For $\mathtt{i}\in \cC$, the identity $1$-morphism in $\cC(\mathtt{i},\mathtt{i})$
is denoted $\mathbbm{1}_{\mathtt{i}}$ and, for a $1$-morphism $\mathrm{F}$, the corresponding identity $2$-morphism
in $\mathrm{Hom}_{\ccC(\mathtt{i},\mathtt{j})}(\mathrm{F},\mathrm{F})$ is denoted  $\mathrm{id}_{\mathrm{F}}$.
We write $\circ$ for composition of $1$-morphisms (which we often omit), 
$\circ_0$ for horizontal composition of $2$-morphisms 
and $\circ_1$ for vertical composition of $2$-morphisms. We set $\mathbf{Cat}$ to be the $2$-category of 
small categories.

\subsection{Finitary $2$-categories}\label{s1.2}

We call an additive $\Bbbk$-linear category {\em finitary} provided that it has split idempotents, finitely 
many isomorphism classes of indecomposable objects and sets of morphisms form finite dimensional 
$\Bbbk$-vec\-tor spaces. We denote by $\mathfrak{A}^f_{\Bbbk}$ the $2$-category whose
\begin{itemize}
\item objects are finitary additive $\Bbbk$-linear categories,
\item $1$-morphisms are additive $\Bbbk$-linear functors,
\item $2$-morphisms are natural transformations of functors.
\end{itemize}

Now we define a {\em finitary} $2$-category (over $\Bbbk$) to be a $2$-category $\cC$ such that
\begin{itemize}
\item it has finitely many objects;
\item for each pair $\mathtt{i},\mathtt{j}$ of objects, 
$\cC(\mathtt{i},\mathtt{j})$ belongs to $\mathfrak{A}_{\Bbbk}^f$ 
and horizontal composition is additive and $\Bbbk$-linear;
\item for every $\mathtt{i}\in\cC$, the corresponding $1$-morphism $\mathbbm{1}_{\mathtt{i}}$ is indecomposable.
\end{itemize}
We refer the reader to \cite{Le,McL} for generalities on abstract $2$-categories and to
\cite{MM1}--\cite{MM5}, \cite{Xa} for more information on finitary $2$-ca\-te\-gories.

\subsection{$2$-representations}\label{s1.3}

For a finitary $2$-category $\cC$, a {\em $2$-representation} of $\cC$ is a strict 
$2$-functor from $\cC$ to the $2$-category $\mathbf{Cat}$. A {\em finitary $2$-representation} of $\cC$ is 
a strict  $2$-functor from $\cC$ to the $2$-category $\mathfrak{A}_{\Bbbk}^f$. We usually denote $2$-representations 
by $\mathbf{M},\mathbf{N},\dots$, in particular, $\mathbf{P}_{\mathtt{i}}$ stands for the
$\mathtt{i}$-th {\em principal} $2$-representation $\cC(\mathtt{i},{}_-)$, where $\mathtt{i}\in\cC$. 
All finitary $2$-representations of $\cC$ form a $2$-category, which we denote by $\cC\text{-}\mathrm{afmod}$.
In this $2$-category, $1$-morphisms are $2$-natural transformations and $2$-morphisms are modifications 
(see \cite{Le,MM3} for details).

Two $2$-representations $\mathbf{M}$ and $\mathbf{N}$ of $\cC$ are said to be {\em equivalent} provided that 
there is a $2$-natural transformation $\Phi:\mathbf{M}\to\mathbf{N}$ such that $\Phi_{\mathtt{i}}$
is an equivalence of categories for each $\mathtt{i}$.

Abusing notation, we write $\mathrm{F}\, X$ instead of  $\mathbf{M}(\mathrm{F})\, X$ for a $1$-morphism
$\mathrm{F}$.

Consider a $2$-representation $\mathbf{M}$ of $\cC$ and assume that $\mathbf{M}(\mathtt{i})$ is additive and
idempotent split for each $\mathtt{i}\in\cC$. For every collection $X_i\in \mathbf{M}(\mathtt{i}_i)$ of 
objects, where $i\in I$, the additive closure $\mathrm{add}(\{\mathrm{F}X_i\})$, where $i\in I$ 
and $\mathrm{F}$ runs  through the set of all $1$-morphisms of $\cC$, has the structure of a $2$-representation 
of $\cC$ by restriction. We write $\mathbf{G}_{\mathbf{M}}(\{X_i:i\in I\})$ for this $2$-subrepresentation 
of $\mathbf{M}$.

\subsection{Combinatorics of finitary $2$-categories}\label{s1.4}

For a finitary $2$-category $\cC$, we denote by $\mathcal{S}(\cC)$ the set of isomorphism classes of 
indecomposable $1$-morphisms 
in $\cC$, which forms a multisemigroup by \cite[Section~3]{MM2}. This multisemigroup comes equipped with 
several natural preorders. For two $1$-morphisms $\mathrm{F}$ and $\mathrm{G}$, we say 
$\mathrm{G}\geq_L\mathrm{F}$ in the {\em left preorder} if there is a $1$-morphism $\mathrm{H}$ such that 
$\mathrm{G}$ appears, up to isomorphism, as a direct summand of $\mathrm{H}\circ \mathrm{F}$. 
A {\em left cell} is an equivalence class for this preorder. Similarly one defines the {\em right} and 
{\em two-sided} preorders $\geq_R$ and $\geq_J$ and the corresponding {\em right} and {\em two-sided} 
cells, respectively.

Note that $\geq_L$ defines a genuine partial order on the set of left cells. A similar statement
holds for $\geq_R$ and right cells, and for $\geq_J$ and two-sided cells.

\subsection{Weakly fiat and fiat $2$-categories}\label{s1.5}

For a $2$-category $\cC$, we denote by $\cC^{\mathrm{op}}$ the $2$-category obtained from 
$\cC$ by reversing both $1$- and $2$-morphisms. 

We call a finitary $2$-category $\cC$ {\em weakly fiat} if 
\begin{itemize}
\item $\cC$ is equipped with a weak equivalence $*:\cC\to \cC^{\mathrm{op}}$;
\item for any pair $\mathtt{i}, \mathtt{j}$ of objects and every $1$-morphism
$\mathrm{F}\in\cC(\mathtt{i},\mathtt{j})$, there are
$2$-mor\-phisms $\alpha:\mathrm{F}\circ\mathrm{F}^*\to
\mathbbm{1}_{\mathtt{j}}$ and $\beta:\mathbbm{1}_{\mathtt{i}}\to
\mathrm{F}^*\circ\mathrm{F}$ with the property that
$\alpha_{\mathrm{F}}\circ_1\mathrm{F}(\beta)=\mathrm{id}_{\mathrm{F}}$ and
$\mathrm{F}^*(\alpha)\circ_1\beta_{\mathrm{F}^*}=\mathrm{id}_{\mathrm{F}^*}$.
\end{itemize}
If $*$ is a weak involution, we call $\cC$ a {\em fiat} $2$-category, see \cite{MM1,MM2}.

The equivalence $*$ sends a $1$-morphism $\mathrm{F}$ to $\mathrm{F}^*$ and the image of
$\mathrm{F}$ under an inverse of $*$ will be denoted ${}^*\mathrm{F}$. Note that both
$\mathrm{F}\mapsto \mathrm{F}^*$ and $\mathrm{F}\mapsto {}^*\mathrm{F}$ are isomorphisms
between $\leq_L$ and $\leq_R$.

We note that the last condition in the paragraph above mirrors the notion of {\em rigidity} in
\cite{EGNO,ENO}, see also references therein.

\subsection{$2$-ideals}\label{s1.6}

Given a a $2$-category $\cC$, a {\em left $2$-ideal} $\cI$ of $\cC$ is a $2$-semicategory
on the same objects as $\cC$ and in which, for each pair $\mathtt{i},\mathtt{j}$ of objects, 
the set $\cI(\mathtt{i},\mathtt{j})$ is an ideal in  $\cC(\mathtt{i},\mathtt{j})$ such that
$\cI$ is closed under left horizontal multiplication with both $1$- and $2$-morphisms in $\cC$. 
{\em Right $2$-ideals} and {\em two-sided $2$-ideals} (also just called {\em $2$-ideals})
are defined similarly. For instance, principal $2$-representations are left $2$-ideals in $\cC$.

For a $2$-category $\cC$ and a $2$-representation $\mathbf{M}$ of $\cC$, an {\em ideal}
$\mathbf{I}$ in $\mathbf{M}$ consists of an ideal $\mathbf{I}(\mathtt{i})$ in 
each $\mathbf{M}(\mathtt{i})$, for $\mathtt{i}\in\cC$, which is stable under the action of $\cC$. 

\subsection{Abelianization}\label{s1.7}

For a finitary additive $\Bbbk$-linear category $\mathcal{A}$, its 
{\em abelianization} is the abelian category $\overline{\mathcal{A}}$, in which objects are
diagrams $X\overset{\eta}{\longrightarrow}Y$, for $X,Y\in \mathcal{A}$ and $\eta\in\mathcal{A}(X,Y)$,
and in which morphisms are equivalence classes of solid commutative diagrams of the form
\begin{displaymath}
\xymatrix{ 
X\ar[rr]^{\eta}\ar[d]_{\tau_1}&&Y\ar[d]^{\tau_2}\ar@{-->}[dll]_{\tau_3}\\
X'\ar[rr]_{\eta'}&&Y'
} 
\end{displaymath}
modulo the equivalence relation $\Bbbk$-linearly spanned by 
diagrams for which there is $\tau_3$ as shown by the dashed 
arrow with $\eta'\tau_3=\tau_2$, see \cite{Fr}. The category $\overline{\mathcal{A}}$
is equivalent to the left module category for the finite dimensional $\Bbbk$-algebra
\begin{displaymath}
\mathrm{End}_{\mathcal{A}}(P)^{\mathrm{op}}\quad  \text{ where }
\quad  P:=\bigoplus_{Q\in\mathrm{Ind}(\mathcal{A})/\cong}Q
\end{displaymath}
where $\mathrm{Ind}(\mathcal{A})$ denotes the collection of indecomposable objects in $\mathcal{A}$.

Given a $2$-category $\cC$ and a finitary $2$-representation $\mathbf{M}$  of $\cC$, the
{\em abelianization} of $\mathbf{M}$ is the $2$-representation $\overline{\mathbf{M}}$ of $\cC$
which sends each $\mathtt{i}\in\cC$ to the category $\overline{\mathbf{M}(\mathtt{i})}$ and
defines the action of $\cC$ on diagrams component-wise. Thus the action of each $1$-morphism 
in $\cC$ on an abelianized finitary $2$-representation is right exact.

\subsection{The $2$-category $\cC_{A}$}\label{s1.8}

Fix a basic, self-injective, non-semisimple, connected, finite dimensional $\Bbbk$-algebra $A$ and a small 
category $\mathcal{A}$ which is equivalent to $A\text{-}\mathrm{mod}$. 
Define, following \cite[Section~7.3]{MM1}, the $2$-category $\cC_{A}$ to have
\begin{itemize}
\item one object $\clubsuit$ (identified with $\mathcal{A}$);
\item direct sums of functors with summands isomorphic to the identity
functor or to tensoring with projective $A\text{-}A$-bimodules as $1$-morphisms; 
\item natural transformations of functors as $2$-morphisms.
\end{itemize}
By a {\em projective functor}, we mean a functor which is isomorphic to tensoring with a projective 
$A\text{-}A$-bimodule.

We fix a set $\{e_1,e_2,\dots,e_n\}$ of primitive, pairwise orthogonal idempotents in $A$ which sum to
the identity element $1$ in $A$, and denote, for $i,j\in\{1,2,\dots,n\}$, by $\mathrm{F}_{ij}$
an indecomposable $1$-morphism given by tensoring with the $A\text{-}A$-bimodule $Ae_i\otimes e_jA$.
Then the $2$-category $\cC_{A}$ has two two-sided cells: the 
minimal one, consisting of  the isomorphism class of the identity morphism, and the maximal
one, which we call $\mathcal{J}$, consisting of the isomorphism classes of 
$\mathrm{F}_{ij}$ for $i,j\in\{1,2,\dots,n\}$. The two-sided cell $\mathcal{J}$ decomposes into
left respectively right cells
\begin{displaymath}
\mathcal{L}_j:=\{\mathrm{F}_{ij}:i\in\{1,2,\dots,n\}\}\quad\text{ and }\quad
\mathcal{R}_i:=\{\mathrm{F}_{ij}:j\in\{1,2,\dots,n\}\},
\end{displaymath}
where $i,j\in\{1,2,\dots,n\}$. We denote by $\sigma:\{1,2,\dots,n\}\to \{1,2,\dots,n\}$ the {\em Nakayama} 
bijection which is defined by requiring
$\mathrm{soc}\,Ae_i\cong \mathrm{top}\,Ae_{\sigma(i)}$ or, equivalently,  
$Ae_i\cong \mathrm{Hom}_{\Bbbk}(e_{\sigma(i)}A,\Bbbk)$. The isomorphism
\begin{displaymath}
\mathrm{Hom}_{A}(Ae_i\otimes_{\Bbbk}e_{j}A,{}_-)\cong
\mathrm{Hom}_{\Bbbk}(e_{j}A,\Bbbk)\otimes_{\Bbbk} e_iA\otimes_A {}_-,
\end{displaymath}
see for example \cite[Section~7.3]{MM1}, implies that $(\mathrm{F}_{ij},\mathrm{F}_{\sigma^{-1}(j)i})$ 
form an adjoint pair of functors. Hence $\cC_{A}$ is weakly fiat, where $*$ is given on $1$-morphisms by
$\mathrm{F}_{ij}^*=\mathrm{F}_{\sigma^{-1}(j)i}$. Moreover, $\cC_{A}$ is fiat if and only if $A$
is weakly symmetric, that is, $\sigma$ is the identity map.

Setting $\displaystyle\mathrm{F}:=\bigoplus_{i,j=1}^n\mathrm{F}_{ij}$ we have
$\mathrm{F}\circ \mathrm{F}\cong \mathrm{F}^{\oplus \dim(A)}$ and $\mathrm{F}^*\cong \mathrm{F}$.

Every nonzero two-sided $2$-ideal in $\cC_A$ necessarily contains the identity $2$-morphism on each 
indecomposable $1$-morphism not isomorphic to the identity, see \cite[Section~3.5]{Ag}, which means that 
$\cC_A$ is {\em $\mathcal{J}$-simple} in the sense of \cite[Section~6.2]{MM2}.

\section{Transitive $2$-representations}\label{s2}

\subsection{Cell $2$-representations}\label{s2.1}

Let $\cC$ be a finitary $2$-category. Given a left cell $\mathcal{L}$ in $\cC$, there is 
$\mathtt{i}=\mathtt{i}_{\mathcal{L}}\in\cC$ with the property that every $1$-morphism in $\mathcal{L}$
has domain $\mathtt{i}$. This allows us to define $\mathbf{N}:=\mathbf{G}_{\mathbf{P}_{\mathtt{i}}}(\mathcal{L})$. 

The $2$-representation $\mathbf{N}$ has a unique maximal ideal $\mathbf{I}$ which does not contain 
$\mathrm{id}_{\mathrm{F}}$ for any $\mathrm{F}\in\mathcal{L}$, see \cite[Lemma~3]{MM5}.
The corresponding quotient $2$-functor $\mathbf{C}_{\mathcal{L}}:=\mathbf{N}/\mathbf{I}$ is said to be 
the {\em (additive) cell $2$-representation} of $\cC$ associated to $\mathcal{L}$. 

\subsection{Strongly regular cells}\label{s2.2}

Let $\cC$ be a weakly fiat $2$-category and $\mathcal{J}$ a two-sided cell in $\cC$. We say that $\mathcal{J}$
is {\em strongly regular}, see \cite[Section~4.8]{MM1}, provided that the following two conditions are satisfied:
\begin{itemize}
\item any two left (resp. right) cells in $\mathcal{J}$ are not comparable with respect to the left (resp. right) order;
\item the intersection between any right and any left cell in $\mathcal{J}$ contains precisely one isomorphism
class of indecomposable $1$-morphisms.
\end{itemize}
If $\mathcal{L}$ is a left cell belonging to a strongly regular two-sided cell, then $\mathcal{L}$ contains
a unique element $\mathrm{G}=\mathrm{G}_{\mathcal{L}}$ such that $\mathrm{G}^{*}\in \mathcal{L}$,
see Proposition~\ref{prop103}. This element is called the {\em Duflo involution} of $\mathcal{L}$. 

If $\mathcal{J}$ is a strongly regular two-sided cell and $\mathrm{F}\in\mathcal{J}$, then the component of ${}^*\mathrm{F}\circ \mathrm{F}$ which belongs to $\mathcal{J}$ is in the same left cell as $\mathrm{F}$ and in the same right cell as $\mathrm{F}^*$, and the intersection of these two cells contains a unique element $\mathrm{H}$. Hence  there is a
positive integer $\mathbf{m}_{\mathrm{F}}$ such that ${}^*\mathrm{F}\circ \mathrm{F}\cong \mathrm{H}^{\oplus \mathbf{m}_{\mathrm{F}}}\oplus \mathrm{K}$, where
no indecomposable direct summand of $\mathrm{K}$ belongs to $\mathcal{J}$.

\subsection{The numerical condition}\label{s2.3}

Let $\cC$ be a weakly fiat $2$-category and $\mathcal{L}$ a left cell in $\cC$ which belongs to a strongly 
regular two-sided cell. Set  $\mathtt{i}=\mathtt{i}_{\mathcal{L}}$ and consider the corresponding cell 
$2$-representation $\mathbf{C}_{\mathcal{L}}$, as defined in Section~\ref{s2.1}, as well as its 
abelianization $\overline{\mathbf{C}}_{\mathcal{L}}$. For each $\mathtt{j}\in\cC$, isomorphism classes 
of indecomposable projective respectively simple objects in $\overline{\mathbf{C}}_{\mathcal{L}}(\mathtt{j})$ 
are indexed by isomorphism classes of $1$-morphisms in $\mathcal{L}\cap \cC(\mathtt{i},\mathtt{j})$ and 
denoted by $P_{\mathrm{F}}$ and $L_{\mathrm{F}}$, respectively. 

By Proposition~\ref{prop102}\eqref{prop102.2} and Proposition~\ref{prop101}\eqref{prop101.1}, respectively, 
for all $\mathrm{F},\mathrm{H}\in \mathcal{L}$ we have
\begin{equation}
\label{eq1}  
\mathrm{F}\, L_{\mathrm{H}}\cong
\begin{cases}
P_{\mathrm{F}} &  \text{if }\mathrm{H}\cong \mathrm{G}_{\mathcal{L}},\\
 0 & \text{otherwise}.
\end{cases}
\end{equation}

\begin{proposition}\label{prop1}
Let $\cC$ be a weakly fiat $2$-category and $\mathcal{J}$ a strongly regular two-sided cell. Then the
function
\begin{displaymath}
\begin{array}{ccc}
\mathcal{J}&\overset{\mathbf{m}}{\longrightarrow}&\mathbb{Z}\\ 
\mathrm{F}&\mapsto &\mathbf{m}_{\mathrm{F}}
\end{array}
\end{displaymath}
is constant on right cells of $\mathcal{J}$.
\end{proposition}

\begin{proof}
Let $\mathcal{L}$ be a left cell in $\mathcal{J}$. Consider the corresponding $2$-representations
$\mathbf{C}_{\mathcal{L}}$ and $\overline{\mathbf{C}}_{\mathcal{L}}$. From \eqref{eq1} 
and \cite[Lemma~13]{MM5}, it follows that each $\mathrm{F}\in \mathcal{L}$ is represented in
$\overline{\mathbf{C}}_{\mathcal{L}}$ by an indecomposable projective functor.

For each $\mathtt{j}\in \cC$, denote by $A_{\mathtt{j}}$ the basic algebra whose module category is
equivalent to $\overline{\mathbf{C}}_{\mathcal{L}}(\mathtt{j})$, and fix some set 
$\{e^{\mathtt{j}}_1,e^{\mathtt{j}}_2,\dots,e^{\mathtt{j}}_{n_{\mathtt{j}}}\}$ of primitive pairwise orthogonal 
idempotents in $A_{\mathtt{j}}$ which sum to the identity. Then, without loss of generality, we may 
assume that the projective functors corresponding to $\mathrm{F}\in \mathcal{L}$ are precisely 
$A_{\mathtt{j}}e^{\mathtt{j}}_k\otimes_{\Bbbk}e^{\mathtt{i}}_1A_{\mathtt{i}}$, where $k=1,2,\dots,n_{\mathtt{j}}$.

By Proposition~\ref{prop104}\eqref{prop104.1}, the algebra $A_{\mathtt{j}}$ is self-injective.
The computation in \cite[Lemma~45]{MM1} shows that $\mathrm{F}^*$ are represented by 
projective functors corresponding to 
$A_{\mathtt{i}}e^{\mathtt{i}}_{\sigma^{-1}(1)}\otimes_{\Bbbk}e^{\mathtt{j}}_kA_{\mathtt{j}}$.
Hence $m_{\mathrm{F}^*}=\dim(e^{\mathtt{i}}_1A_{\mathtt{i}}e^{\mathtt{i}}_{\sigma^{-1}(1)})$ and this clearly does not
depend on the choice of $\mathrm{F}\in \mathcal{L}$. Since $\mathrm{F}\mapsto \mathrm{F}^*$ defines a bijection
from $\mathcal{L}$ to the right cell in $\mathcal{J}$ containing $\mathrm{G}^*$, the function $\mathbf{m}$
is constant on this right cell. The fact that every right cell contains a Duflo involution,
which follows from Proposition~\ref{prop103}\eqref{prop103.3}, completes the proof.
\end{proof}

\begin{remark}\label{rem2}
{\rm
Proposition~\ref{prop1} implies that the additional numerical condition which appears in 
\cite[Theorem~43]{MM1}, \cite[Section~4]{MM3}, \cite[Theorems~14,~16]{MM3} and  \cite[Theorem~18]{MM5}
is redundant.
} 
\end{remark}

\subsection{Transitive and simple transitive $2$-representations}\label{s2.4}

A finitary $2$-representations $\mathbf{M}$ of a finitary $2$-category
$\cC$ is called  {\em transitive} if, for each $\mathtt{i}$ 
and each non-zero object $X\in\mathbf{M}(\mathtt{i})$, we have $\mathbf{G}_{\mathbf{M}}(\{X\})=\mathbf{M}$.
If $\mathbf{M}$ is transitive, then, by \cite[Lemma~4]{MM5}, it has a unique maximal ideal $\mathbf{I}$ 
which does not contain any identity morphisms with the exception of the one for the zero object. 
We call $\mathbf{M}$ {\em simple transitive} if $\mathbf{I}=0$. In general, the quotient  $\underline{\mathbf{M}}$ 
of $\mathbf{M}$ by $\mathbf{I}$ is called the {\em simple transitive quotient} of $\mathbf{M}$. 

If $\cC$ is weakly fiat, then \cite[Proposition~6]{MM5}, \cite[Theorem~43]{MM1} and Section~\ref{s7} show that
all cell $2$-representations associated to left cells of strongly regular two-sided cells are simple transitive.
Furthermore, \cite[Theorem~18]{MM5} and Theorem~\ref{thm106} assert that, under the assumption that all two-sided 
cells of $\cC$ are strongly regular, each simple transitive $2$-representation of $\cC$ is equivalent 
to a cell $2$-representation.

\subsection{Simple transitive subquotients of finitary $2$-representations}\label{s2.5}

Let $\cC$ be a finitary $2$-category.
For a finitary $2$-representation  $\mathbf{M}$ of $\cC$, denote by $\mathrm{Ind}(\mathbf{M})$ the set of
isomorphism classes of indecomposable objects in $\coprod_{\mathtt{i}\in\ccC}\mathbf{M}(\mathtt{i})$.
Assume that $\mathbf{X}\subset \mathrm{Ind}(\mathbf{M})$ is stable under the induced action of $\cC$,  that is 
any indecomposable direct summand of $\mathrm{F}\, X$ belongs to $\mathbf{X}$ for any $X\in \mathbf{X}$
and any $1$-morphism $\mathrm{F}\in \cC$. Then the additive closure in 
$\coprod_{\mathtt{i}\in\ccC}\mathbf{M}(\mathtt{i})$ of all objects $X$ whose isomorphism 
class is in $\mathbf{X}$ inherits the structure of a $2$-representation of $\cC$
by restriction. We will denote this $2$-representation by $\mathbf{M}_{\mathbf{X}}$. 

If $\mathbf{X}\subset \mathbf{Y}\subset \mathrm{Ind}(\mathbf{M})$ are such that both
$\mathbf{M}_{\mathbf{X}}$ and $\mathbf{M}_{\mathbf{Y}}$ are defined and the quotient
$\mathbf{M}_{\mathbf{Y}}/\mathbf{M}_{\mathbf{X}}$ is transitive, then the simple transitive quotient of
$\mathbf{M}_{\mathbf{Y}}/\mathbf{M}_{\mathbf{X}}$ is called a {\em simple transitive subquotient}
of $\mathbf{M}$. By \cite[Theorem~8]{MM5}, the multiset of  simple transitive subquotients
of $\mathbf{M}$ (up to equivalence) is an invariant of $\mathbf{M}$.
For example, for weakly fiat $\cC$ with strongly regular two-sided cells, 
simple transitive subquotients of principal $2$-representations are
exactly cell $2$-representations.

\subsection{Isotypic $2$-representations}\label{s2.6}

A finitary $2$-representation  $\mathbf{M}$ of a finitary $2$-category $\cC$ is called {\em isotypic} provided 
that all simple transitive subquotients of $\mathbf{M}$ are equivalent. For example, all transitive
$2$-representations of $\cC$ are isotypic and so are direct sums of copies of the same transitive $2$-representation.

We now give a more sophisticated example. We first recall the construction of a tensor product of 
$\Bbbk$-linear categories, see \cite[Section~(1.1)]{GK} or \cite[\S~5.18]{KV}. For two finitary $\Bbbk$-linear
categories $\mathcal{A}$ and $\mathcal{B}$, their {\em tensor product} $\mathcal{A}\boxtimes\mathcal{B}$
is the category with objects 
\begin{displaymath}
\bigoplus_{i=1}^n X_i\boxtimes Y_i, \quad \text{ where }\quad X_i\in \mathcal{A}, Y_i\in \mathcal{B}, 
\end{displaymath}
and the morphism space
\begin{displaymath}
\mathrm{Hom}_{\mathcal{A}\boxtimes\mathcal{B}} 
\left(\bigoplus_{i=1}^n X_i\boxtimes Y_i, \bigoplus_{j=1}^m U_j\boxtimes V_j\right)
\end{displaymath}
is given by the matrix 
\begin{displaymath}
\big[\mathrm{Hom}_{\mathcal{A}}(X_i,U_j)\otimes_{\Bbbk}\mathrm{Hom}_{\mathcal{B}}(Y_i,V_j)
\big]_{i=1,2,\dots,n}^{j=1,2,\dots,m}
\end{displaymath}
with the obvious composition. In the above, we use the usual abbreviation
\begin{displaymath}
\bigoplus_{i=1}^n W_i:=W_1\oplus W_2\oplus\dots \oplus W_n, 
\end{displaymath}
which, in particular, means that  $X_1\boxtimes Y_1\oplus X_2\boxtimes Y_2$ and
$X_2\boxtimes Y_2\oplus X_1\boxtimes Y_1$ are different but isomorphic objects.

Let now $\mathbf{M}$ be a finitary $2$-representation of $\cC$ and $\mathcal{A}$ any finitary 
$\Bbbk$-linear category. Define the {\em $\mathcal{A}$-inflation $\mathbf{M}^{\boxtimes \mathcal{A}}$ 
of $\mathbf{M}$} as follows:
\begin{itemize}
\item for $\mathtt{i}\in\cC$, set $\mathbf{M}^{\boxtimes \mathcal{A}}(\mathtt{i}):=
\mathbf{M}(\mathtt{i})\boxtimes \mathcal{A}$;
\item for a $1$-morphism $\mathrm{F}\in\cC(\mathtt{i},\mathtt{j})$, set 
$\mathbf{M}^{\boxtimes \mathcal{A}}(\mathrm{F}):=
\mathbf{M}(\mathrm{F})\boxtimes\mathrm{Id}_{\mathcal{A}}$;
\item for a $2$-morphism $\alpha:\mathrm{F}\to \mathrm{G}$, define
$\mathbf{M}^{\boxtimes \mathcal{A}}(\alpha):=\mathbf{M}(\alpha)\boxtimes\mathrm{id}_{\mathrm{Id}_{\mathcal{A}}}$. 
\end{itemize}

\begin{proposition}\label{prop2}
Any inflation of any transitive $2$-representation $\mathbf{M}$ of $\cC$ is isotypic. 
\end{proposition}

\begin{proof}
Let $\mathcal{A}$ be a finitary $\Bbbk$-linear category. Let $X_1,X_2,\dots,X_n$ and $Y_1,Y_2,\dots,Y_m$ be 
full and irredundant lists of representatives of isomorphism classes of indecomposable objects 
in $\coprod_{\mathtt{i}\in\ccC}\mathbf{M}(\mathtt{i})$ and $\mathcal{A}$, respectively. 

For $k=0,1,2,\dots,m$, let $\mathbf{X}_k$ denote the subset of $\mathrm{Ind}(\mathbf{M}^{\boxtimes \mathcal{A}})$ 
with representatives $X_i\boxtimes Y_j$, where $i=1,2,\dots,n$ and $j=1,2,\dots,k$. Then, 
for $k=1,2,\dots,m$, we have the corresponding $2$-subrepresentation 
$\mathbf{M}^{\boxtimes \mathcal{A}}_{\mathbf{X}_k}$ and the quotient
\begin{displaymath}
\mathbf{M}^{\boxtimes \mathcal{A}}_{\mathbf{X}_k}/\mathbf{M}^{\boxtimes \mathcal{A}}_{\mathbf{X}_{k-1}}.
\end{displaymath}
By transitivity of $\mathbf{M}$, the latter quotient is transitive.

Now we claim that the simple transitive quotient of 
$\mathbf{M}^{\boxtimes \mathcal{A}}_{\mathbf{X}_k}/\mathbf{M}^{\boxtimes \mathcal{A}}_{\mathbf{X}_{k-1}}$
is equivalent to $\underline{\mathbf{M}}$. Without loss of generality, we may assume $k=m$. Denote by
$\mathcal{I}$ the ideal of $\mathcal{A}$ generated by the identity morphisms on all $Y_j$ with $j<m$,
together with all radical morphisms of the local algebra $\mathrm{End}_{\mathcal{A}}(Y_m)$. 
Then 
\begin{displaymath}
\mathbf{I}:=\coprod_{\mathtt{i}\in\ccC}\left(\mathbf{M}(\mathtt{i})\boxtimes \mathcal{I} \right)
\end{displaymath}
is an ideal in $\mathbf{M}^{\boxtimes \mathcal{A}}$ invariant under the action of $\cC$. 
Sending $X_i\boxtimes Y_m$ to $X_i$, for $i=1,2,\dots,n$, defines a $2$-natural transformation from
$\mathbf{M}^{\boxtimes \mathcal{A}}/\mathbf{I}$ to $\mathbf{M}$, which is an equivalence by construction.
Consequently, the simple transitive quotients of $\mathbf{M}^{\boxtimes \mathcal{A}}/\mathbf{I}$
and $\mathbf{M}$ are equivalent.

The above, together with the Weak Jordan-H{\"o}lder Theorem \cite[Theorem~8]{MM5}, shows that any 
simple transitive subquotient of $\mathbf{M}^{\boxtimes \mathcal{A}}$ is equivalent to 
$\underline{\mathbf{M}}$. Therefore, $\mathbf{M}^{\boxtimes \mathcal{A}}$ is isotypic.
\end{proof}

\subsection{Classification of isotypic faithful $2$-representations 
of $\mathcal{J}$-simple weakly fiat $2$-categories}\label{s2.7}

The following is the main result of this paper.

\begin{theorem}\label{thmmain}
Let $\cC$ be a weakly fiat $2$-category with a unique maximal two-sided cell  $\mathcal{J}$ and $\mathcal{L}$ 
a left cell in $\mathcal{J}$. Assume that $\mathcal{J}$ is strongly regular and that $\cC$ is $\mathcal{J}$-simple.
Then any isotypic faithful $2$-representation of $\cC$ is equivalent to an inflation of $\mathbf{C}_{\mathcal{L}}$.
\end{theorem}

There are many examples of $2$-categories which satisfy all assumptions of Theorem~\ref{thmmain},
obtained from the $2$-categories of Soergel bimodules in type $A$, $2$-Kac-Moody algebras and projective 
functors on self-injective finite dimensional algebras. Namely, choosing any left cell in any of these
$2$-categories, and factoring out the annihilator of the corresponding cell $2$-representation, we
obtain a weakly fiat $\mathcal{J}$-simple $2$-category with a unique maximal two-sided cell $\mathcal{J}$, 
and this cell is, moreover, strongly regular.

\section{Proof of Theorem~\ref{thmmain} for $\cC=\cC_A$}\label{s3}

\subsection{Annihilation filtrations}\label{s3.1}

Let $A$ be a connected self-injective finite dimensional $\Bbbk$-algebra and $\cC_A$ the corresponding weakly fiat
$2$-category as defined in Section~\ref{s1.8}. Let $\mathbf{M}$ be a finitary $2$-representation of $\cC_A$
and $\overline{\mathbf{M}}$ its abelianization. For 
$\mathrm{F},\mathrm{G}\in \cC(\clubsuit,\clubsuit)$, $\alpha:\mathrm{F}\to \mathrm{G}$
and $X\in \overline{\mathbf{M}}(\clubsuit)$, we say that {\em $\alpha$ annihilates $X$} provided that 
the linear map $\alpha_X:\mathrm{F}\, X\to \mathrm{G}\, X$ is zero.

\begin{lemma}\label{lem3}
{\hspace{2mm}}

\begin{enumerate}[$($i$)$]
\item\label{lem3.1} If $\alpha$ annihilates $Y$, and $X$ is a subobject of $Y$, then   $\alpha$ annihilates $Y$.
\item\label{lem3.2} If $\alpha$ annihilates $Y$, and $Z$ is a quotient of $Y$, then   $\alpha$ annihilates $Z$.
\item\label{lem3.3} Any $X\in \overline{\mathbf{M}}(\clubsuit)$ has a unique maximal submodule and a unique maximal
quotient which are annihilated by $\alpha$. 
\end{enumerate}
\end{lemma}

\begin{proof}
A short exact sequence $X\hookrightarrow Y\tto Z$ gives rise to the commutative diagram
\begin{displaymath}
\xymatrix{
0\ar[r] & \mathrm{F}\, X\ar[r]\ar[d]_{\alpha_{X}} & \mathrm{F}\, Y\ar[r]\ar[d]_{\alpha_{Y}} 
& \mathrm{F}\, Z \ar[r]\ar[d]_{\alpha_{Z}}& 0\\
0\ar[r] & \mathrm{G}\, X\ar[r] & \mathrm{G}\, Y\ar[r] & \mathrm{G}\, Z \ar[r]& 0
}
\end{displaymath}
with exact rows. If $\alpha_Y$ is zero, then so are $\alpha_X$ and $\alpha_Z$, which proves claims~\eqref{lem3.1}
and \eqref{lem3.2}. 

We prove the second part of claim~\eqref{lem3.3}, the first being proved similarly.
For $i=1,2$, consider short exact  sequences $K_i\hookrightarrow X\tto Y_i$
and assume that $\alpha_{Y_i}=0$ for $i=1,2$. Note that $\alpha_{Y_1\oplus Y_2}=0$ by additivity.
Then the isomorphism theorem implies that  $X/(K_1\cap K_2)$ is a subobject of $Y_1\oplus Y_2$
and hence is annihilated by $\alpha$ thanks to claim~\eqref{lem3.1}.
\end{proof}

For a collection $\boldsymbol{\alpha}:=\{\alpha_1,\alpha_2,\dots,\alpha_n\}$ of $2$-morphisms,
we say that {\em $\boldsymbol{\alpha}$ annihilates $X$} if $\alpha_i$ annihilates $X$ for each $i$.
Similarly, for a finite dimensional vector space of $2$-morphisms, we will say that it 
{\em annihilates} $X$ if some finite generating set of this space annihilates $X$.

Let $\boldsymbol{\alpha}$ be a fixed collection or a finite dimensional vector space of $2$-morphisms.
A filtration
\begin{equation}\label{eq2}
0=X_0\subset X_1\subset \dots \subset X_k=X
\end{equation}
is called an {\em annihilation filtration} for $\boldsymbol{\alpha}$ provided that each subquotient
$X_i/X_{i-1}$ is  annihilated by $\boldsymbol{\alpha}$. An annihilation filtration is called {\em minimal},
if there does not exist any annihilation filtration of strictly smaller length (not necessarily refining
the given one). The length of a minimal annihilation filtration is called the {\em annihilation length}
of $X$ and denoted $\mathrm{al}(X)$.

Set $\mathrm{sub}_{\boldsymbol{\alpha}}^0(X):=0$. Denote by $\mathrm{sub}_{\boldsymbol{\alpha}}(X)=
\mathrm{sub}_{\boldsymbol{\alpha}}^1(X)$ the largest submodule of $X$ annihilated by $\boldsymbol{\alpha}$, 
which exists by Lemma~\ref{lem3}\eqref{lem3.3}. For $i\in\{2,3,\dots\}$,
define $\mathrm{sub}_{\boldsymbol{\alpha}}^i(X)$ as the full preimage of 
$\mathrm{sub}_{\boldsymbol{\alpha}}(X/\mathrm{sub}_{\boldsymbol{\alpha}}^{i-1}(X))$ in $X$.
The filtration
\begin{displaymath}
0=\mathrm{sub}_{\boldsymbol{\alpha}}^0(X)\subset \mathrm{sub}_{\boldsymbol{\alpha}}^1(X)\subset 
\mathrm{sub}_{\boldsymbol{\alpha}}^2(X)\subset \dots
\end{displaymath}
is called the {\em upper} annihilation filtration of $X$.

Set $\mathrm{und}_{\boldsymbol{\alpha}}^0(X):=X$. Denote by $\mathrm{und}_{\boldsymbol{\alpha}}(X)=
\mathrm{und}_{\boldsymbol{\alpha}}^1(X)$ the minimal submodule of $X$ such that 
$X/\mathrm{und}_{\boldsymbol{\alpha}}(X)$ is annihilated by $\boldsymbol{\alpha}$, 
which exists by Lemma~\ref{lem3}\eqref{lem3.3}. For $i\in\{2,3,\dots\}$,
define $\mathrm{und}_{\boldsymbol{\alpha}}^i(X)$ as the minimal submodule of 
$\mathrm{und}_{\boldsymbol{\alpha}}^{i-1}(X)$ such that 
$\mathrm{und}_{\boldsymbol{\alpha}}^{i-1}(X)/\mathrm{und}_{\boldsymbol{\alpha}}^i(X)$ 
is annihilated by $\boldsymbol{\alpha}$, The filtration
\begin{displaymath}
\dots\subset \mathrm{und}_{\boldsymbol{\alpha}}^{2}(X)
\subset \mathrm{und}_{\boldsymbol{\alpha}}^{1}(X)
\subset \mathrm{und}_{\boldsymbol{\alpha}}^{0}(X)=X
\end{displaymath}
is called the {\em lower} annihilation filtration of $X$.

\begin{proposition}\label{prop4}
{\hspace{2mm}}

\begin{enumerate}[$($i$)$]
\item\label{prop4.1} If \eqref{eq2} is an annihilation filtration, then, for any $i=1,2,\dots,k$, we have
\begin{displaymath}
\mathrm{und}_{\boldsymbol{\alpha}}^{k-i}(X)\subset X_i\subset \mathrm{sub}_{\boldsymbol{\alpha}}^i(X).
\end{displaymath}
\item\label{prop4.2} If  \eqref{eq2} is a minimal annihilation filtration, then 
$\mathrm{al}(X)=k$ is also the length of both the upper and the lower annihilation filtration.
\end{enumerate}
\end{proposition}

\begin{proof}
Claim~\eqref{prop4.1} follows directly from the definitions. Claim~\eqref{prop4.2} follows from claim~\eqref{prop4.1}.
\end{proof}

\begin{remark}\label{rem5}
{\rm 
Existence of annihilation filtration does depend on the choice of $X$ and $\boldsymbol{\alpha}$. For example,
choosing $\boldsymbol{\alpha}$ to contain some identity $2$-morphism, it is easy to construct an example where
no annihilation filtration exists (in these cases the upper and the lower annihilation filtrations stabilize
before reaching $X$ or $0$, respectively).
}
\end{remark}

\begin{example}\label{ex501}
{\rm
Consider the natural action of $\cC_A$ on $A\text{-}\mathrm{mod}$ and let 
$\boldsymbol{\alpha}=1\otimes \mathrm{rad}(A^{op})$ be a subspace of $\mathrm{End}_{\ccC_A}(\mathrm{F})$,
see Section~\ref{s1.8}. Then it is easy to check that the upper annihilation filtration of ${}_AA$ 
with respect to $\boldsymbol{\alpha}$ coincides with the socle series of ${}_AA$ and the lower
annihilation filtration of ${}_AA$  with respect to $\boldsymbol{\alpha}$ coincides with 
the radical series of ${}_AA$.
} 
\end{example}

\subsection{Auxiliary statements}\label{s3.2}

We first introduce some notation that will remain in place for the remainder of this section.
Let $A$ be a self-injective finite dimensional $\Bbbk$-algebra and $\cC_A$ the corresponding weakly fiat
$2$-category as defined in Section~\ref{s1.8}. We recall that $\mathrm{F}$ stands for the 
multiplicity free direct sum of all indecomposable projective functors. Let $\mathbf{M}$ be an isotypic 
faithful $2$-representation of $\cC_A$ and $\overline{\mathbf{M}}$ its abelianization.
Isotypicality and  faithfulness of $\mathbf{M}$, together with \cite[Theorem~15]{MM5}, imply that 
all simple transitive subquotients of $\mathbf{M}$ are equivalent to $\mathbf{C}_{\mathcal{L}}$.

Let $P_1,P_2,\dots,P_r$ be a complete and irredundant list of representatives of isomorphism classes of 
indecomposable projective objects in $\overline{\mathbf{M}}(\clubsuit)$ and $L_1,L_2,\dots,L_r$ their
corresponding simple tops. For $i,j=1,2,\dots,r$, let $a_{ij}$ denote the multiplicity of $P_i$ as
a direct summand of $\mathrm{F}\, P_j$,  and let $b_{ij}$ denote the composition multiplicity of $L_i$ in 
$\mathrm{F}\, L_j$. Then, since $\mathrm{F}^*\cong \mathrm{F}$, we have $b_{ij}=a_{ji}$ by \cite[Lemma~10]{MM5} 
and we can set up the matrix  $[\mathrm{F}]_{\mathbf{M}}:=(a_{ij})_{i,j=1,\dots,r}$. We denote by
$[\mathrm{F}]_{\mathbf{C}_{\mathcal{L}}}$ a similarly defined matrix for the cell
$2$-representation $\mathbf{C}_{\mathcal{L}}$ of $\cC_A$ corresponding to the left cell
$\mathcal{L}=\mathcal{L}_1$ as in Section~\ref{s1.8}.

\begin{lemma}\label{lem6}
We can rearrange the ordering of  $P_1,P_2,\dots,P_r$
such that $[\mathrm{F}]_{\mathbf{M}}$ has the diagonal form
\begin{displaymath}
[\mathrm{F}]_{\mathbf{M}}=
\left( 
\begin{array}{cccc}
[\mathrm{F}]_{\mathbf{C}_{\mathcal{L}}}& 0 & \cdots & 0\\
0 &  [\mathrm{F}]_{\mathbf{C}_{\mathcal{L}}}&  \cdots & 0\\
\vdots  & \vdots & \ddots&\vdots\\
0& 0 & \cdots &  [\mathrm{F}]_{\mathbf{C}_{\mathcal{L}}}\\
\end{array}
\right).
\end{displaymath}
\end{lemma}

\begin{proof}
Since $\mathbf{M}$ is isotypic, without loss of generality we may assume that $[\mathrm{F}]_{\mathbf{M}}$ has the form
\begin{equation}\label{eq3}
[\mathrm{F}]_{\mathbf{M}}=
\left( 
\begin{array}{cccc}
[\mathrm{F}]_{\mathbf{C}_{\mathcal{L}}}& * & \cdots & *\\
0 &  [\mathrm{F}]_{\mathbf{C}_{\mathcal{L}}}&  \cdots & *\\
\vdots  & \vdots & \ddots&\vdots\\
0& 0 & \cdots &  [\mathrm{F}]_{\mathbf{C}_{\mathcal{L}}}\\
\end{array}
\right).
\end{equation}

Let us first prove the claim for the smallest possible non-trivial case
\begin{displaymath}
[\mathrm{F}]_{\mathbf{M}}=
\left( 
\begin{array}{cc}
[\mathrm{F}]_{\mathbf{C}_{\mathcal{L}}}& Q \\
0 &  [\mathrm{F}]_{\mathbf{C}_{\mathcal{L}}}
\end{array}
\right).
\end{displaymath}
Since $\mathrm{F}\circ \mathrm{F}\cong \mathrm{F}^{\oplus \dim(A)}$, we see that
\begin{displaymath}
[\mathrm{F}]_{\mathbf{C}_{\mathcal{L}}}Q+Q [\mathrm{F}]_{\mathbf{C}_{\mathcal{L}}}=
\dim(A)Q.
\end{displaymath}
Multiplying with $[\mathrm{F}]_{\mathbf{C}_{\mathcal{L}}}$ on the left and using
$[\mathrm{F}]_{\mathbf{C}_{\mathcal{L}}}^2=\dim(A)[\mathrm{F}]_{\mathbf{C}_{\mathcal{L}}}$, we obtain
\begin{displaymath}
\dim(A)[\mathrm{F}]_{\mathbf{C}_{\mathcal{L}}}Q+[\mathrm{F}]_{\mathbf{C}_{\mathcal{L}}}Q [\mathrm{F}]_{\mathbf{C}_{\mathcal{L}}}=
\dim(A)[\mathrm{F}]_{\mathbf{C}_{\mathcal{L}}}Q,
\end{displaymath}
which implies 
\begin{displaymath}
[\mathrm{F}]_{\mathbf{C}_{\mathcal{L}}}Q [\mathrm{F}]_{\mathbf{C}_{\mathcal{L}}}=0.
\end{displaymath}
As $[\mathrm{F}]_{\mathbf{C}_{\mathcal{L}}}$ is a matrix with positive coefficients and 
$Q$ is a matrix with non-negative coefficients, we get $Q=0$.

The general case follows from the above baby case by double induction on the columns from left to right,
and on the entries within each column from bottom to top.
\end{proof}

Set $\hat{P}:=P_1\oplus P_2\oplus \dots\oplus P_r$ and, for $X\in \overline{\mathbf{M}}(\clubsuit)$, define
\begin{displaymath}
\dim(X):=\dim\mathrm{Hom}_{\overline{\mathbf{M}}(\clubsuit)}(\hat{P},X).
\end{displaymath}
If $L$ is a simple $A$-module, then $\mathrm{F}\, L\cong {}_AA$ and hence
$\dim(\mathrm{F}\, L)=\dim(A)\dim(L)$. By exactness of $\mathrm{F}$, this extends to all
$A$-modules and hence to any object in $\mathbf{C}_{\mathcal{L}}(\clubsuit)$.
Using Lemma~\ref{lem6}, we thus obtain   
\begin{equation}\label{eq5}
\dim(\mathrm{F}\, X)=\dim(A)\dim(X)
\end{equation}
for any $X\in \overline{\mathbf{M}}(\clubsuit)$.

As explained in Section~\ref{s1.8}, the matrix  
$[\mathrm{F}]_{\mathbf{C}_{\mathcal{L}}}$ has size $n\times n$.
We now choose a special ordering of $P_1,P_2,\dots,P_r$. First of all, we need  
$[\mathrm{F}]_{\mathbf{M}}$ to have the form given by Lemma~\ref{lem6}. 
We note that  $r=nk$ for some $k\in\{1,2,\dots\}$. 
For every $i=1,2,\dots,k$, we have a transitive $2$-representation of $\cC_A$ on the additive closure of 
$P_{n(i-1)+1},P_{n(i-1)+2},\dots,P_{ni}$. The simple transitive quotients are equivalent for all $i$.
We choose arbitrarily the ordering of $P_1,P_2,\dots,P_n$ and assume that each such equivalence induces 
simply a shift by $n$ on indices of the indecomposable projectives in these transitive $2$-representations.

Set $P:=P_1\oplus P_2\oplus \dots\oplus P_n$  and define $L:=P/\mathrm{rad}(P)$. 
Define $\boldsymbol{\alpha}$ to be
$1\otimes \mathrm{rad}(A^{op})$ considered as a subspace of the space $\mathrm{End}_{\ccC_A}(\mathrm{F})$
of $2$-endomorphisms of $\mathrm{F}$. Note that $\boldsymbol{\alpha}$ is nilpotent, and hence any object
in $\overline{\mathbf{M}}(\clubsuit)$ has an annihilation filtration with respect to $\boldsymbol{\alpha}$.
Furthermore, since the nilpotency degree of $\boldsymbol{\alpha}$ is exactly the Loewy length
$\mathbf{c}:=\mathrm{ll}({}_AA)$, we have $\mathrm{al}(X)\leq \mathbf{c}$ for any 
$X\in \overline{\mathbf{M}}(\clubsuit)$. We set
\begin{displaymath}
N_i:=P_i/\mathrm{und}_{\boldsymbol{\alpha}}^{1}(P_i) \text{ for } i=1,2,\dots,r, \quad\quad N:=\bigoplus_{i=1}^n N_i,
\end{displaymath}
and note that $N\tto L$ by nilpotency of $\boldsymbol{\alpha}$.

\begin{lemma}\label{lem7}
Let $Q$ be a basic projective generator in $\mathbf{C}_{\mathcal{L}}(\clubsuit)$.
Then $\mathrm{al}(Q)= \mathbf{c}$.
\end{lemma}

\begin{proof}
By \cite[Proposition~9]{MM5}, we may assume that $\mathbf{C}_{\mathcal{L}}(\clubsuit)$ is the defining 
$2$-representation of $\cC_A$ and that  $Q={}_AA$. As $\cC_A$ is $\mathcal{J}$-simple, 
$\mathbf{C}_{\mathcal{L}}$ is faithful and hence $\mathbf{C}_{\mathcal{L}}(\boldsymbol{\alpha})$
has nilpotency degree exactly $\mathbf{c}$. This means that $\mathrm{al}({}_AA)\geq  \mathbf{c}$ implying the claim. 
\end{proof}

\begin{lemma}\label{lem8}
For each $i=1,2,\dots,n$,
we have $\mathrm{al}(\mathrm{F}\, L_i)=\mathrm{al}(\mathrm{F}\, N_i)=\mathrm{al}(\mathrm{F}\, P_i)= \mathbf{c}$.
\end{lemma}

\begin{proof}
As $P_i\tto N_i\tto L_i$, it is enough to prove that $\mathrm{al}(\mathrm{F}\, L_i)=\mathbf{c}$. Note that 
$\mathrm{F}\, L_i\neq 0$ since $\mathrm{F}$ is exact and $\mathbf{M}$ is faithful. This 
induces a finitary $2$-representation, $\mathbf{N}$, of $\cC_A$ on $\mathrm{add}(\mathrm{F}\, L_i)$. 
Since $\mathrm{F}$ does not annihilate $\mathrm{F}\, L_i$, at least one of the simple transitive subquotients
of $\mathbf{N}$ is equivalent to $\mathbf{C}_{\mathcal{L}}$. Using Lemma~\ref{lem7}, we thus
deduce that $\mathrm{F}\, L_i$ has a subquotient $X$ such that $\mathrm{al}(X)= \mathbf{c}$. Therefore, Lemma~\ref{lem3} yields that $\mathrm{al}(\mathrm{F}\, L_i)\geq \mathbf{c}$  and we are done. 
\end{proof}

Let $\mathbf{K}$ denote the restriction of $\mathbf{M}$ to $\mathrm{add}(P)$. Then $\mathbf{K}$ is transitive and
its simple transitive quotient is equivalent to $\mathbf{C}_{\mathcal{L}}$. Denote by $\mathbf{I}$ the
corresponding maximal $\cC_A$-stable ideal of $\mathbf{K}$. Denote by $I$ the subspace of
$\mathrm{End}(P)$ belonging to $\mathbf{I}$. Then \cite[Proposition~9]{MM5} implies that
$\mathrm{End}_{\overline{\mathbf{M}}(\clubsuit)}(P)/I\cong A$. 
For $i=1,2,\dots,n$, we set $\tilde{P}_i:=P_i/(P_i\cap IP)$.

\begin{lemma}\label{lem9} 
For $i,j,s=1,2,\dots,n$, we have 
\begin{displaymath}
\mathrm{F}_{ij}\, L_s\cong
\begin{cases}
\tilde{P}_i & \text{if } s=j, \\
0 & \text{otherwise}. 
\end{cases}
\end{displaymath}

\end{lemma}

\begin{proof}
This follows by comparing the action of $\mathrm{F}_{ij}$ on $\mathbf{K}/\mathbf{I}$ and
on $A\text{-}\mathrm{mod}$, the latter being equivalent to $\mathbf{C}_{\mathcal{L}}$
by \cite[Proposition~9]{MM5}.
\end{proof}

\begin{lemma}\label{lem10}
{\hspace{2mm}}

\begin{enumerate}[$($i$)$]
\item\label{lem10.1} For $i,j,k=1,2,\dots,n$, we have $\mathrm{F}_{ij}\, N_k=0$ unless $j=k$.
\item\label{lem10.2} For $i,j=1,2,\dots,n$, we have $P_{i}\tto \mathrm{F}_{ij}\, N_j$.
\end{enumerate}
\end{lemma}

\begin{proof}
By adjunction,  
\begin{displaymath}
\mathrm{Hom}_{\overline{\mathbf{M}}(\clubsuit)}(\mathrm{F}_{ij}\, N_k, L_s) \cong
\mathrm{Hom}_{\overline{\mathbf{M}}(\clubsuit)}(N_k,\mathrm{F}_{\sigma^{-1}(j)i}\,  L_s),  
\end{displaymath}
and the object $\mathrm{F}_{\sigma^{-1}(j)i}\,  L_s$ cannot have $L_k$ as a composition subquotient unless 
$s\in\{1,2,\dots,n\}$,
moreover, for such $s$ we have $\mathrm{F}_{\sigma^{-1}(j)i}\,  L_s=0$ unless $s=i$. If $s=i$, then,
by Lemma~\ref{lem9}, we have
\begin{displaymath}
\mathrm{Hom}_{\overline{\mathbf{M}}(\clubsuit)}(N_k,\mathrm{F}_{\sigma^{-1}(j)i}\,  L_i) \cong 
\mathrm{Hom}_{\overline{\mathbf{M}}(\clubsuit)}(N_k,\tilde{P}_{\sigma^{-1}(j)}).
\end{displaymath}
As $\boldsymbol{\alpha}$ annihilates $N_k$, the image of $N_k$ under such a homomorphism
is contained in $\mathrm{sub}_{\boldsymbol{\alpha}}(\tilde{P}_{\sigma^{-1}(j)})$, which 
is isomorphic to $L_j$, see Example~\ref{ex501}. Hence
$\mathrm{Hom}_{\overline{\mathbf{M}}(\clubsuit)}(N_k,\tilde{P}_{\sigma^{-1}(j)})$ is zero unless $j=k$,
and in the latter case this space is one-dimensional. Both claims of the lemma follow.
\end{proof}

\begin{lemma}\label{lem11}
The object $N$ is a progenerator for the full subcategory
$\mathrm{sub}_{\boldsymbol{\alpha}}(\overline{\mathbf{M}}(\clubsuit))$ of $\overline{\mathbf{M}}(\clubsuit)$
consisting of all objects isomorphic to $\mathrm{sub}_{\boldsymbol{\alpha}}(X)$ for
some $X\in \overline{\mathbf{M}}(\clubsuit)$.
\end{lemma}

\begin{proof}
Let $X\hookrightarrow Y\tto Z$ be a short exact sequence in 
$\mathrm{sub}_{\boldsymbol{\alpha}}(\overline{\mathbf{M}}(\clubsuit))$. By projectivity of 
$P$, we have a short exact sequence
\begin{equation}\label{eq7}
0\to \mathrm{Hom}_{\overline{\mathbf{M}}(\clubsuit)} (P,X)\to
\mathrm{Hom}_{\overline{\mathbf{M}}(\clubsuit)} (P,Y)\to 
\mathrm{Hom}_{\overline{\mathbf{M}}(\clubsuit)} (P,Z)\to 0.
\end{equation}
As $\boldsymbol{\alpha}$ annihilates $X$, $Y$ and $Z$, any homomorphism from $P$ to any of these objects factors 
over $N$. Therefore \eqref{eq7} induces an exact sequence
\begin{displaymath}
0\to \mathrm{Hom}_{\overline{\mathbf{M}}(\clubsuit)} (N,X)\to
\mathrm{Hom}_{\overline{\mathbf{M}}(\clubsuit)} (N,Y)\to 
\mathrm{Hom}_{\overline{\mathbf{M}}(\clubsuit)} (N,Z)\to 0,
\end{displaymath}
which implies that $N$ is relatively projective in 
$\mathrm{sub}_{\boldsymbol{\alpha}}(\overline{\mathbf{M}}(\clubsuit))$.
As $N\tto L$, it is even a progenerator.
\end{proof}

\subsection{Analysis of $\mathrm{F}\, N_s$ and $\hat{P}$}\label{s3.3}

\begin{proposition}\label{prop12}
For any $s=1,2,\dots,n$, the object $\mathrm{F}\, N_s$ has a filtration of length $\dim(A)$ 
in which all subquotients are isomorphic to $N_i$ for some $i\in\{1,2,\dots,n\}$. Moreover,
each such $N_i$ appears as a subquotient in this filtration $l_i$ times, where $l_i$ is the composition
multiplicity of $L_i$ in ${}_AA$.
\end{proposition}

\begin{proof}
Any composition series 
\begin{equation}\label{eq801}
0=C_0\subset C_1\subset \dots\subset C_{\dim(A)}= {}_AA
\end{equation}
gives rise to the filtration
\begin{equation}\label{eq8}
0=X_0 \subset X_1\subset \dots\subset X_{\dim(A)}= \mathrm{F}\, N_s,
\end{equation}
where $X_i$ is the image of $(C_i\otimes A^{\mathrm{op}})_{N}$ in $\mathrm{F}\, N_s$.
As $\boldsymbol{\alpha}$ annihilates $N_s$, it also annihilates any subquotient $X_i/X_{i-1}$.
Using Lemma~\ref{lem10}\eqref{lem10.2}  and Lemma~\ref{lem11}, we see that each
subquotient $X_i/X_{i-1}$ is a quotient of a sum of copies of $N$.

Now we would like to estimate the number of tops in each $X_i/X_{i-1}$. There is a unique
$t\in\{1,2,\dots,n\}$ such that $Ae_t$ surjects onto $C_i/C_{i-1}$. This induces a morphism of functors
\begin{displaymath}
\bigoplus_{j=1}^n\mathrm{F}_{tj}\to  \mathrm{F}
\end{displaymath}
which, evaluated at $N_s$, induces a surjection $\mathrm{F}_{ts}\, N_s\tto X_i/X_{i-1}$
by Lemma~\ref{lem10}\eqref{lem10.1}.
By  Lemma~\ref{lem10}\eqref{lem10.2}, $\mathrm{F}_{ts}\, N_s$ has simple top $L_t$ and hence 
\begin{equation}\label{eq12}
N_t\tto X_i/X_{i-1}. 
\end{equation}
Therefore the number of times $L_t$ appears as simple top in some subquotient
in filtration~\eqref{eq8} is at most $l_t$. 

Now choose $j$ such that $\dim(N_j)$ is maximal possible. From the above, by additivity, we have
\begin{equation}\label{eq9}
\dim(\mathrm{F}\, N_j)\leq \sum_{t=1}^nl_t\dim(N_t)\leq \sum_{t=1}^nl_t\dim(N_j)=\dim(A)\dim(N_j)
\end{equation}
since $\sum_{t=1}^nl_t=\dim(A)$. If there were $t$ such that $\dim(N_t)<\dim(N_j)$, then
the second inequality in \eqref{eq9} would be strict. At the same time, 
we know that $\dim(\mathrm{F}\, N_j)=\dim(A)\dim(N_j)$ 
by \eqref{eq5}, a contradiction. Therefore $\dim(N_t)=\dim(N_j)$ for all $t$. 

Hence, the total number of simple tops in all subquotients  of filtration \eqref{eq8} is exactly 
$\sum_{t=1}^n l_t=\dim(A)$. By \eqref{eq5} and the above, we have
\begin{displaymath}
\dim(\mathrm{F}\, N_s)=\dim(\mathrm{F}\, N_j)=\dim(A)\dim(N_j).
\end{displaymath}
By comparing dimensions, we see that the surjection \eqref{eq12} is, in fact, an isomorphism.
The claim follows.
\end{proof}

For any two $1$-morphisms $\mathrm{G}$ and $\mathrm{H}$ in $\cC$, and any $2$-morphism 
$\beta:\mathrm{G}\to\mathrm{H}$, we consider the right exact endofunctor 
$\mathrm{Coker}(\overline{\mathbf{M}}(\beta))$ of $\overline{\mathbf{M}}(\clubsuit)$.
For $i,j\in\{1,2,\dots,n\}$, we denote by $\mathrm{Q}_{ij}$ the endofunctor 
$\mathrm{Coker}(\overline{\mathbf{M}}(\beta_{ij}))$ where 
$\beta_{ij}: \mathrm{G}_{ij}\to \mathrm{F}_{ij}$
corresponds to a presentation for the simple quotient of $Ae_i\otimes_{\Bbbk}e_jA$ 
in the category of $A\text{-}A$-bimodules.

\begin{lemma}\label{lem19}
{\hspace{2mm}}

\begin{enumerate}[$($i$)$]
\item\label{lem19.1} For $i,j\in\{1,2,\dots,n\}$, we have
\begin{displaymath}
\mathrm{Q}_{ij}\circ\mathrm{Q}_{kl}\cong
\begin{cases}
\mathrm{Q}_{il} & \text{if } j=k,\\
0 & \text{otherwise}.
\end{cases}
\end{displaymath}
\item\label{lem19.2} For $i,j,s\in\{1,2,\dots,n\}$, we have
\begin{displaymath}
\mathrm{Q}_{ij}\, N_s\cong
\begin{cases}
N_i & \text{if } j=s,\\
0 & \text{otherwise}.
\end{cases}
\end{displaymath}
\item\label{lem19.3} For $i\in\{1,2,\dots,n\}$, the restriction of 
\begin{displaymath}
Q:=\bigoplus_{i=1}^n\mathrm{Q}_{ii} 
\end{displaymath}
to $\mathrm{sub}_{\boldsymbol{\alpha}}(\overline{\mathbf{M}}(\clubsuit))$ is isomorphic to the identity functor.
\end{enumerate}
\end{lemma}

\begin{proof}
The composition $\mathrm{Q}_{ij}\circ\mathrm{Q}_{kl}$ is isomorphic to the cokernel of the map
\begin{displaymath}
\xymatrix{
\mathrm{G}_{ij}\mathrm{F}_{kl}\oplus \mathrm{F}_{ij}\mathrm{G}_{kl}
\ar[rrrr]^{((\beta_{ij})_{\mathrm{F}_{kl}},\mathrm{F}_{ij}(\beta_{kl}))}&&&&
\mathrm{F}_{ij}\mathrm{F}_{kl}, 
}
\end{displaymath}
compare \cite[Section~3.5]{MM1}. On the level of $A\text{-}A$-bimodules, it is easy to check that the
cokernel of $((\beta_{ij})_{\mathrm{F}_{kl}},\mathrm{F}_{ij}(\beta_{kl}))$ is isomorphic to the cokernel of
$\beta_{il}$ if $k=j$ and to zero otherwise. Claim~\eqref{lem19.1} follows.

That $\mathrm{Q}_{ij}\, N_s=0$ for $j\neq s$ follows from Lemma~\ref{lem10}\eqref{lem10.1}.
If $j=s$, then $\mathrm{Q}_{is}\, N_s=N_s$ follows from the proof of Proposition~\ref{prop12}.
This proves claim~\eqref{lem19.2}.

On the level of $A\text{-}A$-bimodules, the quotient of $A$ modulo its radical is isomorphic to the direct sum
of the cokernels of $\beta_{ii}$ for $i=1,2,\dots,n$. This  gives an epimorphic natural transformation
$\gamma:\mathrm{Id}_{\overline{\mathbf{M}}(\clubsuit)}\tto Q$.
For every $s=1,2,\dots,n$,  the epimorphism $\gamma_{N_s}$ is, in fact, an isomorphism due to claim~\eqref{lem19.2}.
As $N$ is a progenerator in $\overline{\mathbf{M}}(\clubsuit)$ by Lemma~\ref{lem11}, claim~\eqref{lem19.3}
follows and the proof is complete.
\end{proof}

\begin{lemma}\label{lem16}
{\hspace{2mm}}

\begin{enumerate}[$($i$)$] 
\item\label{lem16.1}
For $i,s\in\{1,2,\dots,n\}$ such that $\mathrm{al}(\mathrm{F}_{is}\, N_s)=\mathbf{c}$, we have
\begin{displaymath}
\mathrm{und}_{\boldsymbol{\alpha}}(\mathrm{F}_{is}\,N_s)=
\mathrm{sub}^{\mathbf{c}-1}_{\boldsymbol{\alpha}}(\mathrm{F}_{is}\,N_s)\quad\text{ and }\quad
\mathrm{F}_{is}\, N_s/\mathrm{und}_{\boldsymbol{\alpha}}(\mathrm{F}_{is}\,N_s)\cong N_i.
\end{displaymath}
\item\label{lem16.2} 
We have $(\mathrm{F}_{is}\, P)\cong P_i^{\oplus \dim(Ae_{\sigma^{-1}(s)})}$ and hence
\begin{displaymath}
(\mathrm{F}_{is}\, P)/\mathrm{und}_{\boldsymbol{\alpha}}(\mathrm{F}_{is}\,P)
\cong N_i^{\oplus \dim(Ae_{\sigma^{-1}(s)})}.
\end{displaymath}
\end{enumerate}
\end{lemma}

\begin{proof}
From Proposition~\ref{prop12}, we obtain an exact sequence $\mathrm{Ker}\hookrightarrow \mathrm{F}_{is}\,N_s\tto N_i$
and the proof of Proposition~\ref{prop12} shows that 
$\mathrm{Ker}\subset \mathrm{sub}^{\mathbf{c}-1}_{\boldsymbol{\alpha}}(\mathrm{F}_{is}\,N_s)$.
On the other hand, $(\mathrm{F}_{is}\,N_s)/\mathrm{und}_{\boldsymbol{\alpha}}(\mathrm{F}_{is}\,N_s)\tto N_i$
as $N_i$ is annihilated by $\boldsymbol{\alpha}$. By relative projectivity of $N_i$,
see Lemma~\ref{lem11}, it is a direct summand of 
$(\mathrm{F}_{is}\,N_s)/\mathrm{und}_{\boldsymbol{\alpha}}(\mathrm{F}_{is}\,N_s)$ and hence
coincides with the latter as both have simple tops. This proves claim~\eqref{lem16.1}.

To see claim~\eqref{lem16.2}, note that $\mathrm{F}_{is}\, P$ is projective, and that, using adjunction,
\begin{equation}\label{eq161}
\mathrm{Hom}_{\overline{\mathbf{M}}(\clubsuit)}(\mathrm{F}_{is}\, P, L_k) \cong  \mathrm{Hom}_{\overline{\mathbf{M}}(\clubsuit)}(P,\mathrm{F}_{\sigma^{-1}(s)i}\,  L_k) 
\end{equation}
is zero unless $k=i$ by Lemma \ref{lem9}, implying that $\mathrm{F}_{is}\, P$ is a direct sum of copies of $P_i$.  
If $k=i$, the space in \eqref{eq161} is isomorphic to 
$ \mathrm{Hom}_{\overline{\mathbf{M}}(\clubsuit)}(P,\tilde{P}_{\sigma^{-1}(s)})$, 
again by Lemma \ref{lem9}, which has dimension $\dim(\tilde{P}_{\sigma^{-1}(s)}) = \dim(Ae_{\sigma^{-1}(s)})$.  
Both claims in \eqref{lem16.2} follow.
\end{proof}

\begin{proposition}\label{prop14}
We have $\mathrm{F}\, N_s\cong P$ for every $s\in\{1,2,\dots,n\}$.
\end{proposition}

\begin{proof}
Consider the diagram 
\begin{equation}\label{eq120}
\xymatrix{ 
P \ar@{->>}[rr]^{f_1}&& \mathrm{F}\, N_s\ar@{->>}[rr]^{f_2}\ar@{->>}[d]^{g_2} 
&& \mathrm{F}\, L_s\ar@{->>}[d]^{g_3}\\
\mathrm{sub}_{\boldsymbol{\alpha}}(P)\ar@{->>}[rr]_{h_1} &&
\mathrm{sub}_{\boldsymbol{\alpha}}(\mathrm{F}\, N_s)\ar@{->>}[rr]_{h_2} &&
\mathrm{sub}_{\boldsymbol{\alpha}}(\mathrm{F}\, L_s)
}
\end{equation}
in which $f_1$ is an epimorphism given by Lemma~\ref{lem10}\eqref{lem10.2}, and 
$f_2$ is an epimorphism induced by the natural projection $N_s\tto L_s$ using exactness of $\mathrm{F}$.
Comparing the action of $\mathrm{F}$ on $\mathbf{K}/\mathbf{I}$ and
on $A\text{-}\mathrm{mod}$, the latter being equivalent to $\mathbf{C}_{\mathcal{L}}$
by \cite[Proposition~9]{MM5}, and using Example~\ref{ex501}, we see that $\mathrm{F}\, L_s$ has isomorphic
top and socle and that this socle coincides with $\mathrm{sub}_{\boldsymbol{\alpha}}(\mathrm{F}\, L_s)$.
Let $\varphi:A\to A$ be an endomorphism of $A$ which induces and isomorphism between 
the top and the socle of ${}_AA$. Define $g_2:=(\varphi\otimes 1)_{N_s}$ and
$g_3:=(\varphi\otimes 1)_{L_s}$ and let  $h_2$ be the restrictions of $f_2$. Then the
right square of the diagram commutes. The fact that the map $g_3$ is an epimorphism again follows by 
comparing the action of $\mathrm{F}$ on $\mathbf{K}/\mathbf{I}$ and
on $A\text{-}\mathrm{mod}$. This implies that $h_2$ is an epimorphism. Finally,  $g_2$ is an epimorphism
since the top of $\mathrm{F}\, N_s$ maps isomorphically on the top of $\mathrm{F}\, L_s$ by $f_2$,
the latter top maps isomorphically onto $\mathrm{sub}_{\boldsymbol{\alpha}}(\mathrm{F}\, L_s)$
by $g_3$, and the top of $\mathrm{sub}_{\boldsymbol{\alpha}}(\mathrm{F}\, N_s)$
maps isomorphically onto $\mathrm{sub}_{\boldsymbol{\alpha}}(\mathrm{F}\, L_s)$ by $h_2$.

Define $h_1$ to be the restriction of $f_1$.
The isomorphism theorem yields
\begin{displaymath}
\mathrm{sub}_{\boldsymbol{\alpha}}(\mathrm{F}\, L_s)\cong
(\mathrm{sub}_{\boldsymbol{\alpha}}(P)+ IP)/IP\cong
\mathrm{sub}_{\boldsymbol{\alpha}}(P)/(\mathrm{sub}_{\boldsymbol{\alpha}}(P)\cap IP)
\end{displaymath}
whence $h_2h_1$ is an epimorphism. Since the map $h_2$ is just factoring out the radical, the map
$h_1$ must be an epimorphism.

The lower row of \eqref{eq120} belongs to $\mathrm{sub}_{\boldsymbol{\alpha}}(\overline{\mathbf{M}}(\clubsuit))$.
We have $\mathrm{sub}_{\boldsymbol{\alpha}}(\mathrm{F}\, N_s)\in\mathrm{add}(N)$ by Proposition~\ref{prop12}, 
moreover, $\mathrm{sub}_{\boldsymbol{\alpha}}(\mathrm{F}\, N_s)\tto L$ as
$\mathrm{sub}_{\boldsymbol{\alpha}}(\mathrm{F}\, L_s)\tto L$. As the map 
$\mathrm{sub}_{\boldsymbol{\alpha}}(\mathrm{F}\, N_s)\tto L$ induces an isomorphism modulo the radical,
we obtain $\mathrm{sub}_{\boldsymbol{\alpha}}(\mathrm{F}\, N_s)\cong N$.
Since $N$ is relatively projective in $\mathrm{sub}_{\boldsymbol{\alpha}}(\overline{\mathbf{M}}(\clubsuit))$
by Lemma~\ref{lem11}, the object $\mathrm{sub}_{\boldsymbol{\alpha}}(\mathrm{F}\, N_s)\cong N$ splits off
as a direct summand of $\mathrm{sub}_{\boldsymbol{\alpha}}(P)$. 

For $i\in\{1,2,\dots,n\}$, applying $\mathrm{F}_{is}$ to the short exact sequence
\begin{displaymath}
0\to N_s\to P\to \mathrm{Coker}\to 0, 
\end{displaymath}
produces the short exact sequence
\begin{displaymath}
0\to \mathrm{F}_{is}\,N_s\to \mathrm{F}_{is}\,P\to \mathrm{F}_{is}\,\mathrm{Coker}\to 0.
\end{displaymath}
Using Lemma~\ref{lem8} and additivity, there is $i$ such that 
$\mathrm{al}(\mathrm{F}_{is}\,N_s)=\mathrm{al}(\mathrm{F}_{is}\,P)=\mathbf{c}$.

Now we claim that 
\begin{multline}\label{eq22}
\mathrm{und}_{\boldsymbol{\alpha}}(\mathrm{F}_{is}\,N_s)\subset 
\mathrm{und}_{\boldsymbol{\alpha}}(\mathrm{F}_{is}\,P)\cap (\mathrm{F}_{is}\,N_s)\subset\\\subset 
\mathrm{sub}_{\boldsymbol{\alpha}}^{\mathbf{c}-1}(\mathrm{F}_{is}\,P)\cap (\mathrm{F}_{is}\,N_s)\subset 
\mathrm{sub}_{\boldsymbol{\alpha}}^{\mathbf{c}-1}(\mathrm{F}_{is}\,N_s). 
\end{multline}
Indeed, the first inclusion is due to the fact that 
\begin{displaymath}
(\mathrm{F}_{is}\,N_s)/\big(\mathrm{und}_{\boldsymbol{\alpha}}(\mathrm{F}_{is}\,P)\cap (\mathrm{F}_{is}\,N_s)\big)
\hookrightarrow (\mathrm{F}_{is}\,P)/\mathrm{und}_{\boldsymbol{\alpha}}(\mathrm{F}_{is}\,P)
\end{displaymath}
by the isomorphism theorem and the right hand side is annihilated by $\boldsymbol{\alpha}$. Therefore
the left hand side is annihilated by $\boldsymbol{\alpha}$ and is hence a quotient of
$(\mathrm{F}_{is}\,N_s)/\mathrm{und}_{\boldsymbol{\alpha}}(\mathrm{F}_{is}\,N_s)$ by the
universal property of the latter. The second inclusion follows from Proposition~\ref{prop4}\eqref{prop4.1}
and the last inclusion is obvious.
By Lemma~\ref{lem16}, all inclusions in \eqref{eq22} are, in fact, equalities. 

Consider the solid  commutative diagram 
\begin{equation}\label{eq27}
\xymatrix{\mathrm{und}_{\boldsymbol{\alpha}}(\mathrm{F}_{is}\,N_s) \ar@{->}[rr] \ar@{_{(}->}[d]
&&\mathrm{F}_{is}\,N_s \ar@{_{(}->}[d] \ar@{->>}[rr]&& N_i \ar@{->}[d]\\
\mathrm{und}_{\boldsymbol{\alpha}}(\mathrm{F}_{is}\,P)  \ar@{->}[rr]
&&\mathrm{F}_{is}\,P\ar@{->>}[rr] \ar@{.>}@/_8pt/[u]
&& N_i^{\oplus \dim(Ae_{\sigma^{-1}(s)})}\ar@{.>}@/_8pt/[u]
}
\end{equation}
with exact rows given by Lemma \ref{lem16}. Here the middle vertical arrow is the natural monomorphism, which
induces both other vertical arrows. Since all inclusions in \eqref{eq22} are equalities,
in particular $\mathrm{und}_{\boldsymbol{\alpha}}(\mathrm{F}_{is}\,N_s)=
\mathrm{und}_{\boldsymbol{\alpha}}(\mathrm{F}_{is}\,P)\cap (\mathrm{F}_{is}\,N_s)$, the left square of
\eqref{eq27} is a pullback, implying that the morphism between cokernels on the right is a monomorphism.
The object $N_i$ being  indecomposable, a Loewy length argument shows that the cokernel morphism 
$N_i \hookrightarrow N_i^{\oplus \dim(Ae_{\sigma^{-1}(s)})}$ splits as shown by the right dotted arrow.
By projectivity of $\mathrm{F}_{is}\,P$, this induces the middle dotted arrow in \eqref{eq27}.
Since $\mathrm{F}_{is}\,N_s$ has simple top, this dotted arrow is a splitting of the embedding 
$\mathrm{F}_{is}\,N_s \hookrightarrow \mathrm{F}_{is}\,P$. Therefore, $\mathrm{F}_{is}\,N_s$ is a direct 
summand of a projective object and is hence projective itself. 

Now we claim that $\mathrm{F}_{js}\,N_s$ is projective for any $j$. This follows from the facts that 
all $\mathrm{F}_{ab}$ send projectives to projectives (as each of them is left adjoint to an exact functor),
and all $\mathrm{F}_{js}$ appear as direct summand in $\mathrm{F}\mathrm{F}_{is}$.
Taking this into account, an application of Lemma~\ref{lem10} completes the proof.
\end{proof}

For $s=1,2,\dots,n$, set $\hat{N}_{s}:=N_{s}\oplus N_{s+n}\oplus \cdots\oplus N_{n(k-1)+s}$.
Using Proposition~\ref{prop14}, we have 
\begin{equation}\label{eq1513}
\mathrm{F}\, \hat{N}_{s}\cong \hat{P} \quad\text{ for each }s.
\end{equation}
By construction, $\mathrm{End}_{\ccC}(\mathrm{F})\cong A\otimes_{\Bbbk}A^{\mathrm{op}}$, so we may
identify $A$ with the subalgebra $A\otimes 1$ in $\mathrm{End}_{\ccC}(\mathrm{F})$.
Set $B_s:=\mathrm{End}_{\overline{\mathbf{M}}(\clubsuit)}(\hat{N}_s)$. 

\begin{lemma}\label{lem24}
For any $s,t=1,2,\dots,n$ there is an algebra isomorphism $B_s\cong B_t$.
\end{lemma}

\begin{proof}
By Lemma~\ref{lem19}, the functor $\mathrm{Q}_{st}$ induces an equivalence between $\mathrm{add}(\hat{N}_{t})$
and $\mathrm{add}(\hat{N}_{s})$, which implies the claim.
\end{proof}

Thanks to Lemma~\ref{lem24}, we can write $B$ for $B_s$. Since $\boldsymbol{\alpha}$
annihilates $\hat{N}_s$, the natural map $\mathrm{End}_{\ccC}(\mathrm{F})\otimes_{\Bbbk}B\to
\mathrm{End}_{\overline{\mathbf{M}}(\clubsuit)}(\hat{P})$ factors through the morphism
\begin{equation}\label{eq29}
A\otimes_{\Bbbk}B\to \mathrm{End}_{\overline{\mathbf{M}}(\clubsuit)}(\hat{P}). 
\end{equation}

\begin{proposition}\label{prop23}
The morphism in \eqref{eq29} is an isomorphism. 
\end{proposition}

\begin{proof}
First we show that the dimensions of both algebras agree. The algebra on the right hand side has dimension
$\dim(\hat{P})$ which is equal to $\dim(A)\dim(\hat{N}_s)$ by the calculation in \eqref{eq9}. Since
$\hat{N}_s$ is a basic progenerator in $B\text{-}\mathrm{mod}$, we have $\dim(\hat{N}_s)=\dim(B)$.

It remains to show that the morphism in \eqref{eq29} is injective. Choose a basis 
$\{\varphi_1,\varphi_2,\dots,\varphi_{\dim(A)}\}$ in $A$ compatible with 
the filtration in \eqref{eq801}. In analogy to \eqref{eq8}, we obtain a filtration
\begin{displaymath}
0=Y_1\subset Y_1\subset \dots \subset Y_{\dim(A)}=\mathrm{F}\, \hat{N}_s=\hat{P},
\end{displaymath}
where $Y_i$ is the image of $(C_i\otimes A^{\mathrm{op}})_{\hat{N}_s}$. For a non-zero $\psi\in B$ consider the 
endomorphism $\varphi_i\otimes \psi$ of $\hat{P}$. The proof of Proposition~\ref{prop12} shows that
the quotient of the image of $\varphi_i\otimes \psi$ modulo $Y_{i-1}$ coincides with the image of 
$Q_{ts}(\psi)$ considered as an endomorphism of $\hat{N}_t$. By Lemma~\ref{lem19}, the latter is non-zero.
This implies the claim.
\end{proof}

\subsection{The proof for $\cC=\cC_A$}\label{s3.4}

\begin{proof}[Proof of Theorem~\ref{thmmain} for $\cC_A$.]
For a fixed $s\in\{1,2,\dots,n\}$ set $\mathcal{B}:=\mathrm{add}(\hat{N}_s)$. Consider the inflation
$\mathbf{P}_{\clubsuit}\boxtimes \mathcal{B}$ and define the $2$-natural transformation
\begin{displaymath}
\Phi: \mathbf{P}_{\clubsuit}\boxtimes \mathcal{B}\to \overline{\mathbf{M}}
\end{displaymath}
in the following way:
\begin{itemize}
\item send $\mathrm{G}\boxtimes X$ to $\mathrm{G}\, X$ for $X\in \mathcal{B}$ and a $1$-morphism
$\mathrm{G}\in\cC_A$;
\item send $\alpha\otimes f:\mathrm{G}_1\boxtimes X_1\to\mathrm{G}_2\boxtimes X_2$ to 
$\mathrm{G}_2(f)\circ\alpha_{X_1}=\alpha_{X_2}\circ \mathrm{G}_1(f)$.
\end{itemize}
It is easy to check that this is a well-defined functor and, moreover, defines a $2$-natural transformation.

For our fixed left cell $\mathcal{L}=\mathcal{L}_1$, consider the $2$-subrepresentation $\mathbf{N}$ of
$\mathbf{P}_{\clubsuit}$ as defined in Section~\ref{s2.1}. Restricting $\Phi$ to $\mathbf{N}\boxtimes
\mathcal{B}$ and using \eqref{eq1513}, we get a $2$-natural transformation from 
$\mathbf{N}\boxtimes \mathcal{B}$ to $\mathbf{G}_{\overline{\mathbf{M}}}(\mathrm{add}(\hat{P}))$.
Since the ideal $\mathbf{I}$ defined in Section~\ref{s2.1} is generated by $\boldsymbol{\alpha}$
(see, for example, the proof of \cite[Lemma~14]{MM5}), the ideal $\mathbf{I}$ 
annihilates $\hat{N}_s$ by construction. Hence we get the induced $2$-natural transformation
\begin{displaymath}
\overline{\Phi}: \mathbf{C}_{\mathcal{L}}\boxtimes \mathcal{B}\to 
\mathbf{G}_{\overline{\mathbf{M}}}(\mathrm{add}(\hat{P})).
\end{displaymath}
Proposition~\ref{prop23} shows that $\overline{\Phi}$ is an equivalence of categories. This completes the proof.
\end{proof}

\section{Proof of Theorem~\ref{thmmain} in the general case}\label{s4}

\subsection{Obvious generalization of Section~\ref{s3}}\label{s4.1}

Let $A$ be as in Section~\ref{s1.8}. For certain central subalgebras $X$ in $A$, in \cite[Section~4.5]{MM3}
we constructed certain fiat $2$-subcategories $\cC_{A,X}$ in $\cC_A$.  The only difference between
$\cC_{A,X}$ and $\cC_A$ is that the endomorphism algebra of the identity $1$-morphism in $\cC_{A,X}$ is $X$.
Similarly one defines $\cC_{A,X}$ in the weakly fiat case.
Since we did not use the endomorphism algebra of the identity $1$-morphism in our arguments in Section~\ref{s3}, 
the claim of Theorem~\ref{thmmain} is true for $2$-representations of $\cC_{A,X}$.

Let $A$ be as in Section~\ref{s1.8}, but not necessarily connected, say
\begin{displaymath}
A=A_1\oplus A_2\oplus \dots \oplus A_m. 
\end{displaymath}
As above, choose some central subalgebras $X_i$ in each $A_i$. With such $A$ and $X=(X_1,X_2,\dots,X_m)$ 
one associates a $2$-category $\cC_{A,X}$ with $m$ objects defined similarly to the connected case, see
\cite[Section~7.3]{MM1} and \cite[Section~4.5]{MM3} for details. The arguments in Section~\ref{s3}
work mutatis mutandis in this more general setup.

\subsection{The proof}\label{s4.2}

\begin{proof}[Proof of Theorem~\ref{thmmain}.]
Let $\cC$ be a weakly fiat $2$-category with a unique maximal two-sided cell  $\mathcal{J}$.
Assume that $\mathcal{J}$ is strongly regular and that $\cC$ is $\mathcal{J}$-simple.
Let $\mathbf{M}$ be an isotypic faithful $2$-representation of $\cC$.

Denote by $\cC_\mathcal{J}$ the $2$-full $2$-subcategory of $\cC$ whose $1$-morphisms are those in 
$\mathcal{J}$ as well as the identity $1$-morphisms on their respective sources and targets. 
Restricting $\mathbf{M}$ to $\cC_\mathcal{J}$ produces a faithful $2$-representation, which we 
denote by $\mathbf{M}_\mathcal{J}$. 

Let $\mathcal{L}$ be a left cell in $\mathcal{J}$. Our assumptions and Theorem~\ref{thm107} imply that
all simple transitive subquotients of $\mathbf{M}$ are equivalent to $\mathbf{C}_{\mathcal{L}}$,
in particular, they are not annihilated by any $1$-morphism in $\mathcal{J}$. By Theorem~\ref{thm106},
the $2$-category $\cC_\mathcal{J}$ has only one (up to equivalence) simple transitive $2$-representation
which is not annihilated by $1$-morphisms in $\mathcal{J}$. This means that $\mathbf{M}_\mathcal{J}$
is an isotypic $2$-representation of $\cC_\mathcal{J}$.

By Theorem~\ref{thm105}, the $2$-category $\cC_\mathcal{J}$ is a $2$-category of 
the form $\cC_{A,X}$ for (not necessarily connected) $A$ and $X$ as in Section~\ref{s4.1}.
Therefore we can construct the object $\hat{N}_s$ as in Section~\ref{s3}. 
Using the principal $2$-representation of $\cC$ instead of the principal $2$-representation of $\cC_\mathcal{J}$,
the proof is now completed by the same argument as in Section~\ref{s3.4}. 
\end{proof}

\section{New examples}\label{s5}

\subsection{Tensor product of finitary $2$-categories}\label{s5.1}

The tensor product of categories as described in Section~\ref{s2.6} can be used to produce
new examples of finitary and fiat $2$-categories. This is motivated by the $2$-category of Soergel bimodules
associated to Coxeter systems with several irreducible components. 

Let $\cA$ and $\cC$ be two finitary $2$-categories. Define the $2$-category $\cA\boxtimes\cC$ whose objects are
pairs $(\mathtt{i},\mathtt{j})$ where $\mathtt{i}$ is an object of
$\cA$ and $\mathtt{j}$ is an object of $\cC$, and whose morphism categories are given by
\begin{displaymath}
(\cA\boxtimes\cC)\big((\mathtt{i},\mathtt{j}),(\mathtt{k},\mathtt{l})\big):=
\cA(\mathtt{i},\mathtt{k})\boxtimes\cC(\mathtt{j},\mathtt{l}).
\end{displaymath}
Composition is defined component-wise and
$\mathbbm{1}_{(\mathtt{i},\mathtt{j})} =\mathbbm{1}_{\mathtt{i}}\boxtimes\mathbbm{1}_{\mathtt{j}}$.

\begin{proposition}\label{prop51}
The $2$-category $\cA\boxtimes\cC$ is finitary, moreover, it is (weakly) fiat if both $\cA$ and $\cC$ are.
\end{proposition}

\begin{proof}
As the tensor product of two local algebras is local,  the tensor product of an indecomposable $1$-morphisms
in $\cA$ and an indecomposable $1$-morphisms in $\cC$ will have a local endomorphism algebra. This implies that
$(\cA\boxtimes\cC)\big((\mathtt{i},\mathtt{j}),(\mathtt{k},\mathtt{l})\big)$ is idempotent split. All
other axioms of finitary $2$-categories follow directly from the definition.

If both $\cA$ and $\cC$ are (weakly) fiat, we can define $*$ on $\cA\boxtimes\cC$ component-wise.
Since all compositions are defined component-wise, adjunction morphisms for $\cA\boxtimes\cC$ are
constructed by tensoring adjunction morphisms for $\cA$ with adjunction morphisms for $\cC$.
\end{proof}

If $\mathcal{L}_1$ is a left cell in $\cA$ and $\mathcal{L}_2$ is a left cell in $\cC$, then
\begin{displaymath}
\mathcal{L}_1\boxtimes \mathcal{L}_2 :=\{\mathrm{F}\boxtimes \mathrm{G}\,:\, 
\mathrm{F}\in \mathcal{L}_1,\mathrm{G}\in \mathcal{L}_2 \} 
\end{displaymath}
is a left cell in $\cA\boxtimes\cC$. The partial order on the set of left cells in $\cA\boxtimes\cC$ is the
direct product of the partial orders on the sets of left cells in $\cA$ and $\cC$. Similar statements hold
for right and two-sided cells. It is worth pointing out that the direct product of strongly regular 
two-sided cells is strongly regular.
Furthermore, if $\mathcal{J}_1$ and $\mathcal{J}_2$ are two-sided cells in $\cA$ respectively  $\cC$  
such that $\cA$ and $\cC$ are $\mathcal{J}_1$- respectively $\mathcal{J}_2$-simple, then 
$\cA\boxtimes\cC$ is $\mathcal{J}_1\boxtimes\mathcal{J}_2$-simple.

As a special case of the above construction, consider two self-injective algebras $A$ and $B$. The
finitary $2$-category $\cC_A\boxtimes\cC_B$ can be realized as an ``extension'' of $\cC_{A\otimes B}$ by
adding the (additive closure of) $1$-morphisms given by $A\otimes B\text{-}A\otimes B$-bimodules
\begin{displaymath}
A\otimes_{\Bbbk}(B\otimes_{\Bbbk}B)\quad\text{ and }\quad
(A\otimes_{\Bbbk}A)\otimes_{\Bbbk}B
\end{displaymath}
and all possible natural transformations between all functors.

\subsection{Isotypic $2$-representations for certain tensor products}\label{s5.15}

Let $\cC$ be as in Theorem~\ref{thmmain}  and $\mathcal{J}$ the unique maximal two-sided cell in $\cC$.
For  a local commutative finite dimensional $\Bbbk$-algebra $B$ denote by $\cY_B$ the finitary $2$-category having
\begin{itemize}
\item one object $\spadesuit$;
\item direct sums of copies of $\mathbbm{1}_{\spadesuit}$ as $1$-morphisms; 
\item $B$ as the endomorphism algebra of $\mathbbm{1}_{\spadesuit}$.
\end{itemize}

Consider the $2$-category $\cC\boxtimes \cY_B$. This $2$-category has  
$\hat{\mathcal{J}}:=\mathcal{J}\boxtimes \mathbbm{1}_{\spadesuit}$ as its unique maximal two-sided cell.
Note that $\cC\boxtimes \cY_B$ is not $\hat{\mathcal{J}}$-simple unless $B\cong\Bbbk$.
Given a finitary $2$-representation $\mathbf{M}$ of $\cC$ and a finitary $2$-representation $\mathbf{N}$ of
$\cY_B$, the tensor product $\mathbf{M}\boxtimes \mathbf{N}$ carries the natural structure of a
$2$-representation of $\cC\boxtimes \cY_B$ defined component-wise.

A finitary $2$-representation $\mathbf{N}_{\mathcal{C},\varphi}$ of $\cY_B$ is given, up to equivalence, 
by the data of a finitary category $\mathcal{C}$ and a fixed homomorphism $\varphi$ from $B$ to the 
center of $\mathcal{C}$.

\begin{proposition}\label{prop112}
Every isotypic $2$-representation of  $\cC\boxtimes \cY_B$ which does not annihilate any
$1$-morphism in $\hat{\mathcal{J}}$ is equivalent to $\mathbf{C}_{\mathcal{L}}\boxtimes 
\mathbf{N}_{\mathcal{C},\varphi}$ for some left cell $\mathcal{L}$ in $\mathcal{J}$ and
appropriate $\mathcal{C}$ and $\varphi$ as above.
\end{proposition}

\begin{proof}
Let $\mathbf{M}$ be an isotypic $2$-representation of  $\cC\boxtimes \cY_B$ which does not 
annihilate any $1$-morphism in $\hat{\mathcal{J}}$. Restricting $\mathbf{M}$ to the canonical copy of 
$\cC$ in $\cC\boxtimes \cY_B$ gives a faithful isotypic $2$-representation of $\cC$. Consider 
the category $\mathrm{add}(\hat{N}_s)$. It carries the natural structure of a
$2$-representation of $\cY_B$ by restriction (since any indecomposable $1$-morphism in $\cY_B$ is 
isomorphic to the identity). We denote the resulting $2$-representation of $\cY_B$ by $\mathbf{N}$.
By the above remark, $\mathbf{N}$ is equivalent to $\mathbf{N}_{\mathcal{C},\varphi}$ for some 
$\mathcal{C}$ and $\varphi$.

Let $\mathcal{L}$ be a left cell in  $\mathcal{J}$ and $\mathtt{i}:=\mathtt{i}_{\mathcal{L}}$.
Consider now the $2$-representation $\mathbf{P}_{\mathtt{i}}\boxtimes\mathbf{N}$ of
$\cC\boxtimes \cY_B$. Mapping $\mathrm{F}\boxtimes X$ to $\mathrm{F}\, X$ defines a 
$2$-natural transformation from $\mathbf{P}_{\mathtt{i}}\boxtimes\mathbf{N}$ to
$\mathbf{M}$. Similarly to Section~\ref{s3.4}, this induces an equivalence between 
$\mathbf{C}_{\mathcal{L}}\boxtimes\mathbf{N}$ and $\mathbf{M}$.
\end{proof}

\subsection{Finitary $2$-categories associated with trivial extensions}\label{s5.2}

Let $B$ be a basic connected finite dimensional algebra and let $M$ be a $B\text{-}B$-bimodule in
$\mathrm{add}\{B,B\otimes_{\Bbbk}B\}$. Denote by $A$ the {\em trivial extension} of $B$ by $M$, that is,
the set 
\begin{displaymath}
\left\{ 
\left(\begin{array}{cc}b&m\\0&b\end{array}\right)\,:\,b\in B,m\in M
\right\}
\end{displaymath}
with the obvious (matrix) multiplication. Denote by $\cC(B,M)$ the $2$-category obtained by adding to $\cC_A$
all functors given by tensoring with bimodules in $\mathrm{add}(A\otimes_B A)$ and all natural
transformations between all our functors.

\begin{proposition}\label{prop52}
The $2$-category $\cC(B,M)$ is finitary.
\end{proposition}

\begin{proof}
The only thing we need to check is that $\cC(B,M)$ is closed with respect to composition of $1$-morphisms.
Since ${}_BA_B\cong B\oplus M\in \mathrm{add}\{B,B\otimes_{\Bbbk}B\}$, we have
\begin{displaymath}
(A\otimes_B A)\otimes_A(A\otimes_B A)\cong (A\otimes_B A)\oplus (A\otimes_B M\otimes_B A)
\end{displaymath}
and 
\begin{displaymath}
A\otimes_B M\otimes_B A\in\mathrm{add}\{A\otimes_B A,A\otimes_{\Bbbk} A\} 
\end{displaymath}
by our choice of $M$. 

Next we claim that both $(A\otimes_B A)\otimes_A(A\otimes_{\Bbbk} A)$ and
$(A\otimes_{\Bbbk} A)\otimes_A(A\otimes_B A)$ are projective $A\text{-}A$-bimodules.
We prove the first part and the second part is similar. We have to show that $A\otimes_B A$ is projective
as a left $A$-module. As
\begin{displaymath}
A\otimes_B A\cong A\otimes_B (B\oplus M)\in \mathrm{add}\{A\otimes_B B,A\otimes_B B\otimes_{\Bbbk}B\},
\end{displaymath}
the claim follows and the proof of the proposition is complete.
\end{proof}

This construction generalizes in the obvious way to the case of a not necessarily connected algebra $B$.

\subsection{Fiat $2$-categories associated to iterated trivial extensions of symmetric algebras}\label{s5.3}

\begin{proposition}\label{prop53}
Let $B$ be a connected finite dimensional symmetric algebra and $M=B=B^{*}$.
Then $\cC(B,M)$ is a fiat $2$-category.
\end{proposition}

We will need the following observation.

\begin{lemma}\label{lem55}
The $A\text{-}A$-bimodule $A\otimes_B A$ is indecomposable.
\end{lemma}

\begin{proof}
Denote by $Z$ the center of $B$. Note that $Z$ is local as $B$ is connected. Using adjunction on both sides, we
see 
\begin{displaymath}
\mathrm{Hom}_{A\text{-}A}(A\otimes_B A,A\otimes_B A)\cong
\mathrm{Hom}_{B\text{-}B}(B,A\otimes_B A)\cong Z^{\oplus 4}
\end{displaymath}
since $A\otimes_B A\cong B^{\oplus 4}$ as a $B\text{-}B$-bimodule and $Z\cong \mathrm{End}_{B\text{-}B}(B)$.

Next, a direct calculation shows that the subalgebra
\begin{displaymath}
\hat{Z}:=
\left\{
\left(\begin{array}{cc}a&b\\0&a\end{array}\right)\,:\, a,b\in Z 
\right\} 
\end{displaymath}
of $A$ is central. We denote by $\hat{Z}_1$ the subalgebra of $\hat{Z}$ consisting of diagonal elements and
by $\hat{Z}_2$ the subalgebra of $\hat{Z}$ consisting of upper triangular elements. Thus
$\hat{Z}=\hat{Z}_1\oplus\hat{Z}_2$ as a direct sum of two subspaces. We also note that $\hat{Z}_1$ is isomorphic 
to $Z$.

The commutative algebra $\tilde{Z}:=\hat{Z}\otimes_{\hat{Z}_1}\hat{Z}$ acts faithfully by endomorphisms of the 
$A\text{-}A$-bimodule $A\otimes_B A$ where the left component acts by multiplication on the left and 
the right component acts by multiplication on the right. This induces an isomorphism
$\hat{Z}\otimes_{\hat{Z}_1}\hat{Z}\cong \mathrm{End}_{A\text{-}A}(A\otimes_B A)$ since both are
isomorphic to $Z^{\oplus 4}$ as $Z$-modules.

Finally, note that both $\hat{Z}_2\otimes_{\hat{Z}_1} \hat{Z}$ and $\hat{Z}\otimes_{\hat{Z}_1} \hat{Z}_2$ 
are nilpotent in 
$\tilde{Z}$ of degree $2$ and thus belong to the radical of $\tilde{Z}$. The quotient
$\tilde{Z}/(\hat{Z}_2\otimes_{\hat{Z}_1} \hat{Z}+\hat{Z}\otimes_{\hat{Z}_1} \hat{Z}_2)$ 
is isomorphic to $Z$ and is thus local.
This implies that $\tilde{Z}$ is local, completing the proof.
\end{proof}

\begin{proof}[Proof of Proposition~\ref{prop53}.]
First we note that the trivial extension $A$ of $B$ by $B^*$ is symmetric, in particular, weakly symmetric,
so the $2$-subcategory $\cC_A$ of $\cC(B,M)$ is fiat. In view of Lemma~\ref{lem55}, it only remains to check that 
$A\otimes_B A$ is self-adjoint. This follows from the chain of isomorphisms
\begin{displaymath}
\begin{array}{rcll}
\mathrm{Hom}_{A\text{-}}(A\otimes_B A,A)
&\cong & \mathrm{Hom}_{B\text{-}}(A,A)&\\ 
&\cong & \mathrm{Hom}_{B\text{-}}(A,B)\otimes_{B}A& \text{(as ${}_BA\cong B\oplus B$)}\\ 
&{\cong} & \mathrm{Hom}_{\Bbbk}(A,\Bbbk)\otimes_{B}A& \text{(as $B^*\cong B$)}\\ 
&{\cong} & A\otimes_{B}A& \text{(as $A^*\cong A$)}.
\end{array}
\end{displaymath}
\end{proof}

Proposition~\ref{prop53} implies that for a symmetric algebra $B$ the construction can be iterated.
Set $A_0:=B$ and for  $k=1,2,\dots$ defined $A_k$ recursively as the trivial extension of $A_{k-1}$
by $A_{k-1}^*\cong A_{k-1}$. This leads to a fiat $2$-category $\cC_{B,k}$ which is obtained by
adding to $\cC_{A_k}$ all functors given by tensoring with bimodules in 
$\mathrm{add}(A_k\otimes_{A_i}A_k)$ for all $i<k$ and all natural
transformations between all our functors. The arguments in the proof of  Proposition~\ref{prop53}
generalize mutatis mutandis to show that $\cC_{B,k}$ is a fiat $2$-category. Clearly,
all two-sided cells in $\cC_{B,k}$ are strongly regular.

Combining this construction with the one proposed in Section~\ref{s5.1}, we obtain numerous
new examples of fiat $2$-categories in which all two-sided cells are strongly regular.

\section{Appendix: weakly fiat potpourri}\label{s7}

In this section, we collect appropriate reformulations of statements from \cite{MM1,MM2,MM3}  
(originally for fiat $2$-categories) in the more general setup of weakly fiat $2$-categories.
Let $\cC$ be a weakly fiat $2$-category as defined in Section~\ref{s1.5} with a weak equivalence $*$.

We write $\mathbf{P}$ for the direct sum of all principal representations of $\cC$.
Isomorphism classes of indecomposable projective and simple objects in $\overline{\mathbf{P}}$ are indexed
by elements in $\mathcal{S}(\cC)$ and denoted by $\hat{P}_{\mathrm{F}}$ and 
$\hat{L}_{\mathrm{F}}$. 

\begin{proposition}\label{prop101}
Let $\mathrm{F},\mathrm{G},\mathrm{H}\in \mathcal{S}(\cC)$.

\begin{enumerate}[$($i$)$]
\item\label{prop101.1} The inequality
$\mathrm{F}\,\hat{L}_{\mathrm{G}}\neq 0$ is equivalent to ${}^*\mathrm{F}\leq_R \mathrm{G}$ and also to 
$\mathrm{F}\leq_L \mathrm{G}^*$.
\item\label{prop101.2} The inequality
$[\mathrm{F}\,\hat{L}_{\mathrm{G}}:\hat{L}_{\mathrm{H}}]\neq 0$ implies $\mathrm{H}\leq_L \mathrm{G}$.
\item\label{prop101.3} If $\mathrm{H}\leq_L \mathrm{G}$, then there is $\mathrm{K}\in \mathcal{S}(\cC)$
such that $[\mathrm{K}\,\hat{L}_{\mathrm{G}}:\hat{L}_{\mathrm{H}}]\neq 0$.
\item\label{prop101.4} If $\hat{L}_{\mathrm{F}}$ occurs in the top or socle of $\mathrm{H}\,\hat{L}_{\mathrm{G}}$,
then $\mathrm{F}$ is in the same left cell as $\mathrm{G}$.
\item\label{prop101.5} If $\mathrm{F}\in\cC(\mathtt{i},\mathtt{j})$, then there is a unique (up to scalar)
non-zero homomorphism $\hat{P}_{\mathbbm{1}_{\mathtt{i}}}\to \mathrm{F}^*\,\hat{L}_{\mathrm{F}}$, in particular, 
$\mathrm{F}^*\,\hat{L}_{\mathrm{F}}\neq 0$.
\end{enumerate}
\end{proposition}

\begin{proof}
Mutatis mutandis \cite[Lemmata~12-15]{MM1}. 
\end{proof}

\begin{proposition}\label{prop102}
Let $\mathcal{L}$ be a left cell of $\cC$ and $\mathtt{i}=\mathtt{i}_{\mathcal{L}}$.

\begin{enumerate}[$($i$)$]
\item\label{prop102.1}
There is a unique submodule $K$ of  $\hat{P}_{\mathbbm{1}_{\mathtt{i}}}$ such that
\begin{enumerate}[$($a$)$]
\item\label{prop102.1.1} every simple subquotient of 
$\hat{P}_{\mathbbm{1}_{\mathtt{i}}}/K$ is annihilated by any
$\mathrm{F}\in\mathcal{L}$;
\item\label{prop102.1.2} the module $K$ has simple top 
$\hat{L}_{\mathrm{G}_{\mathcal{L}}}$ for some 
$\mathrm{G}_{\mathcal{L}}\in\mathcal{S}(\mathcal{C})$ and
$\mathrm{F}\, \hat{L}_{\mathrm{G}_{\mathcal{L}}}\neq 0$
for any $\mathrm{F}\in\mathcal{L}$.
\end{enumerate}
\item\label{prop102.2} For any $\mathrm{F}\in\mathcal{L}$
the module $\mathrm{F}\, \hat{L}_{\mathrm{G}_{\mathcal{L}}}$
has simple top $\hat{L}_{\mathrm{F}}$.
\item\label{prop102.3} We have $\mathrm{G}_{\mathcal{L}}\in\mathcal{L}$.
\item\label{prop102.4} For any $\mathrm{F}\in\mathcal{L}$
we have ${}^*\mathrm{F}\leq_R \mathrm{G}_{\mathcal{L}}$
and $\mathrm{F}\leq_L \mathrm{G}_{\mathcal{L}}^*$.
\item\label{prop102.5} We have $\mathrm{G}^*_{\mathcal{L}}\in 
\mathcal{L}$.
\end{enumerate}
\end{proposition}

\begin{proof}
Mutatis mutandis \cite[Proposition~17]{MM1}. 
\end{proof}

The element $\mathrm{G}_{\mathcal{L}}\in\mathcal{L}$ is called the {\em Duflo involution} in $\mathcal{L}$.

\begin{proposition}\label{prop103}
Let $\mathcal{J}$ be a strongly regular two-sided cell in $\cC$ and $\mathcal{L}$ a left cell in $\mathcal{J}$.
Then for $\mathrm{G}\in \mathcal{L}$ the following assertions are equivalent.
\begin{enumerate}[$($a$)$]
\item\label{prop103.1} $\mathrm{G}=\mathrm{G}_{\mathcal{L}}$. 
\item\label{prop103.2} $\mathrm{G}^*\in\mathcal{L}$. 
\item\label{prop103.3} $\{\mathrm{G}\}=\mathcal{L}\cap {}^*\mathcal{L}$. 
\item\label{prop103.4} $\mathrm{G}={}^*\mathrm{H}$ where $\{\mathrm{H}\}=\mathcal{L}\cap \mathcal{L}^*$. 
\item\label{prop103.5} $\mathrm{G}\,\hat{L}_{\mathrm{G}}\neq 0$. 
\end{enumerate}
\end{proposition}

\begin{proof}
We have
\begin{itemize}
\item \eqref{prop103.1}$\Rightarrow$\eqref{prop103.2} by Proposition~\ref{prop102}\eqref{prop102.5};
\item \eqref{prop103.2}$\Rightarrow$\eqref{prop103.4} as $|\mathcal{L}\cap \mathcal{L}^*|=1$;
\item \eqref{prop103.4}$\Rightarrow$\eqref{prop103.3} applying $\mathrm{F}\mapsto {}^*\mathrm{F}$;
\item \eqref{prop103.3} $\Rightarrow$\eqref{prop103.2} applying  $\mathrm{F}\mapsto \mathrm{F}^*$;
\item \eqref{prop103.2}$\Rightarrow$\eqref{prop103.1} again as $|\mathcal{L}\cap \mathcal{L}^*|=1$.
\end{itemize}
Finally, \eqref{prop103.2}$\Leftrightarrow$\eqref{prop103.5} follows from strong regularity of 
$\mathcal{J}$ and Proposition~\ref{prop101}\eqref{prop101.1}.
\end{proof}

Let $\mathcal{J}$ be a strongly regular two-sided cell in $\cC$ and $\mathcal{L}$ a left cell in $\mathcal{J}$.
To simplify computational expressions, we assume that $\mathcal{J}$ is a maximal two-sided cell.
In the cell $2$-representation $\mathbf{C}_{\mathcal{L}}$ as defined in Section~\ref{s2.1}, we have indecomposable
projective objects $P_{\mathrm{F}}$, indecomposable injective  objects $I_{\mathrm{F}}$ and  simple 
objects $L_{\mathrm{F}}$, where $\mathrm{F}\in \mathcal{L}$. Set $\mathrm{G}=\mathrm{G}_{\mathcal{L}}$.

\begin{lemma}\label{lem10305}
For any $\mathrm{F}\in \mathcal{L}$ we have $\mathrm{F}\mathrm{G}\cong \mathbf{m}_{\mathrm{G}}\mathrm{F}$.
\end{lemma}

\begin{proof}
Strong regularity of $\mathcal{J}$ implies that $\mathrm{G}\mathrm{G}=m\mathrm{G}$ 
for some non-negative integer $m$. Applying this to $L_{\mathrm{G}}$,
using exactness of $\mathrm{G}$ and a character argument, we see that 
\begin{displaymath}
m=[P_{\mathrm{G}}:L_{\mathrm{G}}]=\dim\mathrm{End}(P_{\mathrm{G}}). 
\end{displaymath}
At the same time, $\dim\mathrm{End}(P_{\mathrm{G}})=
\dim\mathrm{Hom}(\mathrm{G}\,L_{\mathrm{G}},\mathrm{G}\,L_{\mathrm{G}})$ 
equals $\mathbf{m}_{\mathrm{G}}$ by adjunction.
Thus $m=\mathbf{m}_{\mathrm{G}}$.

Strong regularity of $\mathcal{J}$ again yields $\mathrm{F}\mathrm{G}\cong k\mathrm{F}$ for some non-negative
integer $k$. Moreover, $k$ is, in fact, positive, since $\mathrm{F}\mathrm{G}\, L_{\mathrm{G}}\neq 0$.
Now, computing $\mathrm{F}\mathrm{G}\mathrm{G}$ in two different ways using associativity, and then dividing by $k$,
we obtain $k=\mathbf{m}_{\mathrm{G}}$.
\end{proof}

\begin{proposition}\label{prop104}
Let $\mathrm{F}\in \mathcal{L}$ and $\mathrm{H}\in \mathcal{J}$.
\begin{enumerate}[$($i$)$]
\item\label{prop104.1} The projective object $P_{\mathrm{F}}$ is injective.
\item\label{prop104.2} The object $\mathrm{H}\, L_{\mathrm{F}}$, when non-zero, has a non-zero
projective-injective summand.
\item\label{prop104.3} We have $\mathrm{F}^*L_{\mathrm{F}}=I_{\mathrm{G}}$.
\item\label{prop104.4} The object $\mathrm{H}\, L_{\mathrm{F}}$ is both projective and injective.
\item\label{prop104.5} The object $\mathrm{H}\, L_{\mathrm{F}}$ is indecomposable or zero.
\end{enumerate}
\end{proposition}

\begin{proof}
By adjunction, $L_{\mathrm{G}}$ injects into $\mathrm{F}^*L_{\mathrm{F}}$. Consider an injective object 
$I$ and let $L_{\mathrm{K}}$ be a simple quotient of $I$, where $\mathrm{K}\in \mathcal{L}$. Then
$L_{\mathrm{G}}$ is a subquotient of the injective object $\mathrm{K}^*\, I$. As $\mathcal{J}$ is strongly regular,
it follows from Proposition~\ref{prop101}\eqref{prop101.1} that $\mathrm{G}$ annihilates all simples except for
$L_{\mathrm{G}}$. Hence $L_{\mathrm{G}}$ appears in the top of the injective object 
$\mathrm{G}\mathrm{K}^*\, I$. Applying $\mathrm{F}$, we obtain that $P_{\mathrm{F}}$ is a quotient (and hence
a direct summand) of the injective object $\mathrm{F}\mathrm{G}\mathrm{K}^*\, I$ and is therefore injective. 
This proves claim~\eqref{prop104.1}.

Assume that $\mathrm{H}\, L_{\mathrm{F}}\neq 0$. The argument in the previous paragraph shows that 
$\mathrm{H}'\, L_{\mathrm{F}}$ has a non-zero projective-injective summand for some $\mathrm{H}'$ in the same
left cell as $\mathrm{H}$. Claim~\eqref{prop104.2} is now deduced by multiplying the latter on the left with elements in
$\mathcal{J}$ and  using strong regularity of $\mathcal{J}$.

Using adjunction and strong regularity of $\mathcal{J}$, we see that $\mathrm{F}^*L_{\mathrm{F}}$ has simple
socle $L_{\mathrm{G}}$. Therefore claim~\eqref{prop104.3} follows from claim~\eqref{prop104.2}. 
Claim~\eqref{prop104.4} is implied by claim~\eqref{prop104.3} as any inequality 
$\mathrm{H}\, L_{\mathrm{F}}\neq 0$ means,
by strong regularity of $\mathcal{J}$, that $\mathrm{H}$ and $\mathrm{F}^*$ are in the same left cell and
hence $\mathrm{H}\, L_{\mathrm{F}}$ is a direct summand of the projective-injective object
$\mathrm{K}\mathrm{F}^*L_{\mathrm{F}}$ for some $\mathrm{K}$.

It remains to prove claim~\eqref{prop104.5}. By strong regularity of $\mathcal{J}$, there is a unique
element $\tilde{\mathrm{G}}$ in the right cell of $\mathrm{F}$ such that 
$\tilde{\mathrm{G}}\, L_{\mathrm{F}}\neq 0$. This implies that 
$\tilde{\mathrm{G}}\, \hat{L}_{\mathrm{F}}\neq 0 $ and thus
$\tilde{\mathrm{G}}\, \hat{L}_{\tilde{\mathrm{G}}}\neq 0$ by Proposition~\ref{prop101}\eqref{prop101.1},
as $\tilde{\mathrm{G}}$ and $\mathrm{F}$ are in the same right cell. In particular, 
by Proposition~\ref{prop103}, $\tilde{\mathrm{G}}$ is the Duflo involution in its left cell.

Using claim~\eqref{prop104.4} and strong regularity of $\mathcal{J}$, we deduce 
$\tilde{\mathrm{G}}\, L_{\mathrm{F}}\cong k P_{\mathrm{F}}$ for some non-negative integer $k$.
We compute $\mathrm{F}^*\tilde{\mathrm{G}}\, L_{\mathrm{F}}$ in two different ways. On the one hand,
\begin{displaymath}
\mathrm{F}^*\tilde{\mathrm{G}}\, L_{\mathrm{F}}\cong
k\mathrm{F}^*\,P_{\mathrm{F}}\cong 
k\mathrm{F}^*\mathrm{F}\,L_{\mathrm{F}}\cong
k\mathbf{m}_{\mathrm{F}}P_{\mathrm{G}^*},
\end{displaymath}
where the last isomorphism follows from the isomorphism $\mathrm{F}^*\mathrm{F}\cong ({}^*\mathrm{F}\mathrm{F})^*$
and the fact that $\{\mathrm{G}^*\}=\mathcal{L}\cap\mathcal{L}^*$. On the other hand,
$\mathrm{F}^*$ is in the same left cell as $\tilde{\mathrm{G}}^*$ and hence in the same 
left cell as $\tilde{\mathrm{G}}$ as the latter is the Duflo involution. Therefore, Lemma~\ref{lem10305}
and claim~\eqref{prop104.3} give
\begin{displaymath}
\mathrm{F}^*\tilde{\mathrm{G}}\, L_{\mathrm{F}}\cong
\mathbf{m}_{\tilde{\mathrm{G}}}P_{\mathrm{G}^*}.
\end{displaymath}
From Proposition~\ref{prop1}, we obtain $\mathbf{m}_{\tilde{\mathrm{G}}}=\mathbf{m}_{\mathrm{F}}$ 
which yields $k=1$ and proves claim~\eqref{prop104.5} in the case $\mathrm{H}=\tilde{\mathrm{G}}$.

Now, in the general case, assume $\mathrm{H}\, L_{\mathrm{F}}\cong k P_{\mathrm{K}}$ for some
$\mathrm{K}\in\mathcal{L}$ and a positive integer $k$
(note that $\mathrm{K}$ and $\mathrm{H}$ are then in the same right cell). 
Then  ${}^*\mathrm{H}\mathrm{H}\, L_{\mathrm{F}}\neq 0$
by adjunction and hence ${}^*\mathrm{H}\mathrm{H}\cong \mathbf{m}_{\mathrm{H}}\tilde{\mathrm{G}}$. 
Let us now compute the
dimension of $\mathrm{End}(\mathrm{H}\, L_{\mathrm{F}})$. On the one hand, by adjunction and the fact that
$\tilde{\mathrm{G}}\, L_{\mathrm{F}}\cong P_{\mathrm{F}}$, proved in the previous paragraph, this dimension
equals $\mathbf{m}_{\mathrm{H}}$. On the other hand, it equals 
$k^2\dim\mathrm{End}(P_{\mathrm{K}})$ which, in turn, by adjunction, equals $k^2\mathbf{m}_{\mathrm{K}}$.
Proposition~\ref{prop1} implies $\mathbf{m}_{\mathrm{K}}=\mathbf{m}_{\mathrm{H}}$ and hence $k=1$,
completing the proof.
\end{proof}

Proposition~\ref{prop104} shows that indecomposable $1$-morphisms in $\mathcal{J}$ act, under 
$\mathbf{C}_{\mathcal{L}}$, as indecomposable projective functors for some self-injective algebra.
In analogy to \cite[Theorem~43]{MM1}, this implies the following.

\begin{theorem}\label{thm107} 
Let $\cC$ be a weakly fiat $2$-category and $\mathcal{J}$ a strongly regular  two-sided cell in $\cC$.
\begin{enumerate}[$($i$)$]
\item\label{thm107.1} For any left cell $\mathcal{L}$ in $\mathcal{J}$, the cell $2$-representation
$\mathbf{C}_{\mathcal{L}}$ is strongly simple in the sense of \cite[Section~6.2]{MM1}.
\item\label{thm107.2} If $\mathcal{L}$ and $\mathcal{L}'$ 
are two right cells in $\mathcal{J}$, then the cell $2$-representations
$\mathbf{C}_{\mathcal{L}}$ and $\mathbf{C}_{\mathcal{L}'}$
are equivalent.
\end{enumerate}
\end{theorem}

Using this, and following the arguments in \cite{MM3}, we can generalize \cite[Theorem~13]{MM3}
to all weakly fiat $2$-categories (in the notation of Section~\ref{s4}).

\begin{theorem}\label{thm105}
Let $\cC=\cC_{\mathcal{J}}$ be a skeletal weakly fiat $\mathcal{J}$-simple $2$-category for a 
strongly regular $\mathcal{J}$. Then $\cC$ is biequivalent to $\cC_{A,X}$ for appropriate 
self-injective $A$ and $X\subset Z(A)$.
\end{theorem}

Similarly, we obtain the following generalization of \cite[Theorem~8]{MM5}.

\begin{theorem}\label{thm106}
Let $\cC$ be a weakly fiat $2$-category such that all two-sided cells in $\cC$ are strongly regular. 
Then any simple transitive $2$-representation of $\cC$ is equivalent to a cell $2$-representation.
\end{theorem}


\noindent
Volodymyr Mazorchuk, Department of Mathematics, Uppsala University,
Box 480, SE-751 06, Uppsala, SWEDEN,\\ {\tt mazor\symbol{64}math.uu.se};
http://www.math.uu.se/$\tilde{\hspace{1mm}}$mazor/.

\noindent
Vanessa Miemietz, School of Mathematics, University of East Anglia,\\
Norwich NR4 7TJ, UK, \\ {\tt v.miemietz\symbol{64}uea.ac.uk};
http://www.uea.ac.uk/$\tilde{\hspace{1mm}}$byr09xgu/.


\begin{thebibliography}{999999}
\bibitem[Ag]{Ag} T.~Agerholm. Simple $2$-representations and classification of categorifications. 
PhD Thesis, Aarhus University, 2011.
\bibitem[BFK]{BFK} J.~Bernstein, I.~Frenkel, M.~Khovanov. A categorification of the Temperley-Lieb algebra and Schur 
quotients of $U(\mathfrak{sl}_2)$ via projective and Zuckerman functors. Selecta Math. (N.S.) {\bf 5} (1999), no. 
2, 199--241. 
\bibitem[BG]{BG} J.~Bernstein, S.~Gelfand. Tensor products of 
finite- and infinite-dimensional representations of semisimple Lie 
algebras.  Compositio Math.  {\bf 41}  (1980), no. 2, 245--285. 
\bibitem[CR]{CR} J.~Chuang, R.~Rouquier. Derived equivalences for 
symmetric groups and $\mathfrak{sl}_2$-ca\-te\-go\-ri\-fi\-ca\-ti\-on. 
Ann. of Math.  (2) {\bf 167} (2008), no. 1, 245--298. 
\bibitem[EW]{EW} B.~Elias, G.~Williamson. The Hodge theory of Soergel bimodules.
Preprint arXiv:1212.0791. To appear in Ann. of Math.
\bibitem[EL]{EL} 
A.~Ellis, A.~Lauda. An odd categorification of quantum $\mathfrak{sl}(2)$. Preprint arXiv:1307.7816.
\bibitem[EGNO]{EGNO} P.~Etingof, S.~Gelaki, D.~Nikshych, A
V.~Ostrik. Tensor categories. Manuscript, available from
http://www-math.mit.edu/$\sim$etingof/tenscat1.pdf
\bibitem[ENO]{ENO} P.~Etingof, D.~Nikshych, V.~Ostrik. On fusion categories. Ann. of Math. (2) {\bf 162} 
(2005), no. 2, 581--642.
\bibitem[Fr]{Fr} P.~Freyd. Representations in abelian categories. 
Proc. Conf. Categorical Algebra (1966), 95--120.
\bibitem[GK]{GK} N.~Ganter, M.~Kapranov. Symmetric and exterior powers of categories. Transform. Groups 
{\bf 19} (2014), no. 1, 57--103.
\bibitem[KV]{KV} M.~Kapranov, V.~Voevodsky. $2$-categories and Zamolodchikov tetrahedra equations. 
Algebraic groups and their generalizations: quantum and infinite-dimensional methods (University Park, PA, 1991), 
177--259, Proc. Sympos. Pure Math., {\bf 56}, Part 2, Amer. Math. Soc., Providence, RI, 1994.
\bibitem[KL]{KL} M.~Khovanov, A.~Lauda. A categorification of a
quantum $\mathfrak{sl}_n$. Quantum Topol. {\bf 1} (2010), 1--92.
\bibitem[Le]{Le}  T.~Leinster. Basic bicategories. Preprint arXiv:math/9810017.
\bibitem[LW]{LW}
I.~Losev, B.~Webster. On uniqueness of tensor products of irreducible categorifications.
Preprint arXiv:1303.1336.
\bibitem[McL]{McL} S.~Mac Lane. Categories for the working mathematician. 
Second edition. Graduate Texts in Mathematics, {\bf 5}. Springer-Verlag, New York, 1998.
\bibitem[MT]{MT}
M.~Mackaay, A.-L.~Thiel. Categorifications of the extended affine Hecke algebra and the 
affine $q$-Schur algebra $S(n,r)$ for $2 < r < n$. Preprint arXiv:1302.3102.
\bibitem[MM1]{MM1} V.~Mazorchuk, V.~Miemietz. Cell $2$-representations of finitary
$2$-categories. Compositio Math. {\bf 147} (2011), 1519--1545.
\bibitem[MM2]{MM2} V.~Mazorchuk, V.~Miemietz. Additive versus abelian $2$-representations of 
fiat $2$-ca\-te\-go\-ri\-es. Moscow Math. J. {\bf 14} (2014), no. 3, 595--615.
\bibitem[MM3]{MM3} V.~Mazorchuk, V.~Miemietz. Endomorphisms of cell $2$-representations. 
Preprint arXiv:1207.6236.
\bibitem[MM4]{MM4} V.~Mazorchuk, V.~Miemietz. Morita theory for finitary $2$-categories. 
Preprint arXiv:1304.4698. To appear in Quantum Topol.
\bibitem[MM5]{MM5} V.~Mazorchuk, V.~Miemietz. Transitive $2$-representations of finitary $2$-categories. 
Preprint arXiv:1404.7589. 
\bibitem[Ro]{Ro} R.~Rouquier. $2$-Kac-Moody algebras. Preprint arXiv:0812.5023. 
\bibitem[Ro2]{Ro2} R.~Rouquier. Quiver Hecke algebras and $2$-Lie algebras. Algebra Colloquium {\bf 19} (2012), 359--410.
\bibitem[SS]{SS}
A.~Sartori, C.~Stroppel. Categorification of tensor product representations of $\mathfrak{sl}(k)$ 
and category $\mathcal{O}$. Preprint arXiv:1407.4267.
\bibitem[So]{So} W.~Soergel. The combinatorics of Harish-Chandra bimodules. 
J. Reine Angew. Math. {\bf 429} (1992), 49--74. 
\bibitem[Xa]{Xa} Q.~Xantcha. Gabriel $2$-Quivers for Finitary $2$-Categories.
Preprint arXiv:1310.1586, to appear in J. London Math. Soc.
\end{thebibliography}
\end{document}